\documentclass[10pt,leqno]{article}
\usepackage{verbatim}
\usepackage{amssymb}
\usepackage{mathtools}
\usepackage{amsrefs}
\usepackage{rotating}
\usepackage{amsmath}
\usepackage{tabularx}
\usepackage{theorem}
\usepackage[matrix,tips,frame,color,line,poly,curve]{xy}

\newtheorem{theorem}[equation]{Theorem}
\newtheorem{corollary}[equation]{Corollary}
\newtheorem{definition}[equation]{Definition}
\newtheorem{lemma}[equation]{Lemma}

\newtheorem{proposition}[equation]{Proposition}

{\theorembodyfont{\rmfamily}
\newtheorem{remarkplain}[equation]{Remark}

\newtheorem{exampleplain}[equation]{Example}

}

\renewcommand{\sec}[1]{\section{#1}
\renewcommand{\theequation}{\thesection.\arabic{equation}}
  \setcounter{equation}{0}}
\newcommand{\subsec}[1]{\subsection{#1}
\renewcommand{\theequation}{\thesubsection.\arabic{equation}}
  \setcounter{equation}{0}}

\def\danger{\begin{trivlist}\item[]\noindent%
\begingroup\hangindent=3pc\hangafter=-2
\def\par{\endgraf\endgroup}%
\hbox to0pt{\hskip-\hangindent\dbend\hfill}\ignorespaces}
\def\enddanger{\par\end{trivlist}}

\newcommand{\qed}{\hfill $\square$ \medskip}
\newenvironment{proof}[1][Proof]{\noindent\textbf{#1.} }{\qed}

\newcommand{\Aut}{\mathrm{Aut}}

\newcommand{\diag}{\mathrm{diag}}

\newcommand{\Out}{\mathrm{Out}}
\newcommand{\Int}{\mathrm{Int}}
\renewcommand{\int}{\mathrm{int}}

\newcommand{\Ad}{\mathrm{Ad}}
\newcommand{\zinv}{\mathrm{inv}}
\newcommand{\SRF}{\mathrm{SRF}}
\newcommand{\Gad}{G_\mathrm{ad}}
\newcommand{\Gsc}{G_\mathrm{sc}}
\newcommand{\Zsc}{Z_\mathrm{sc}}
\newcommand{\Ztor}{Z_\mathrm{tor}}
\newcommand{\Gbar}{\overline G}

\newcommand{\Gal}{\mathrm{Gal}}
\newcommand{\Norm}{\mathrm{Norm}}
\newcommand{\Cent}{\mathrm{Cent}}

\renewcommand{\O}{\mathcal O}
\newcommand{\R}{\mathbb R}
\newcommand{\C}{\mathbb C}
\newcommand{\Z}{\mathbb Z}

\newcommand{\Ztwo}{\mathbb Z_2}

\renewcommand{\H}{\mathbb H}
\newcommand{\h}{\mathfrak h}
\renewcommand{\sl}{\mathfrak s\mathfrak l}
\renewcommand{\P}{\mathfrak p}

\newcommand{\B}{\mathcal B}
\renewcommand{\k}{\mathfrak k}

\newcommand{\ch}[1]{#1^\vee}
\newcommand\sigmaqc{\sigma_{\text{qc}}}
\newcommand\thetaqc{\theta_{\text{qc}}}

\newcommand{\cl}{\mathit{cl}}
\newcommand{\Lie}{\mathrm{Lie}}
\newcommand{\opp}{\text{-opp}}

\newcommand{\su}{\mathfrak s\mathfrak u}
\newcommand{\g}{\mathfrak g}
\newcommand\inv{^{-1}}
\newcommand\wh{\widehat}
\newcommand{\GL}{\text{GL}}
\newcommand{\SL}{\text{SL}}
\newcommand{\SO}{\text{SO}}
\newcommand{\SU}{\text{SU}}
\newcommand{\Spin}{\text{Spin}}

\begin{document}
\title{Galois and Cartan Cohomology of Real Groups}
\author{Jeffrey Adams and Olivier Ta\"ibi}
\maketitle

{\renewcommand{\thefootnote}{} 
\footnote{2000 Mathematics Subject Classification: 11E72 (Primary), 20G10, 20G20}
\footnote{Jeffrey Adams is supported in part by  National Science
Foundation Grant \#DMS-1317523}
\footnote{Olivier Ta\"ibi is supported by ERC Starting Grant 306326.}}

\section*{Abstract}
Suppose $G$ is a complex, reductive algebraic group. A real form of $G$ is an antiholomorphic involutive
automorphism $\sigma$, so $G(\R)=G(\C)^\sigma$ is a real Lie group.
Write $H^1(\sigma,G)$ for the Galois cohomology (pointed) set 
$H^1(\text{Gal}(\C/\R),G)$.  A Cartan involution for $\sigma$ is an
involutive holomorphic automorphism $\theta$ of $G$, commuting with
$\sigma$, so that $\theta\sigma$ is a compact real form of $G$.  Let
$H^1(\theta,G)$ be the set $H^1(\Ztwo,G)$ where the action of the
nontrivial element of $\Ztwo$ is by $\theta$.  By analogy with the
Galois group we refer to $H^1(\theta,G)$ as Cartan cohomology of $G$
with respect to $\theta$.  Cartan's classification of real forms of a
connected group, in terms of their maximal compact subgroups, amounts to an isomorphism
$H^1(\sigma,\Gad)\simeq H^1(\theta,\Gad)$ where $\Gad$ is the adjoint
group.  Our main result is a generalization of this: there is a
canonical isomorphism $H^1(\sigma,G)\simeq H^1(\theta,G)$. 

We apply this result to give simple proofs of some well
known structural results: the Kostant-Sekiguchi correspondence of
nilpotent orbits; Matsuki duality of orbits on the flag
variety; conjugacy classes of Cartan subgroups; and structure of the
Weyl group.  We also use it to compute $H^1(\sigma,G)$ for all simple,
simply connected groups, and  to give a cohomological
interpretation of strong real forms. For  the applications it 
is important that we do not assume $G$ is connected.

\sec{Introduction}

Suppose $G$ is a complex, reductive algebraic group.  A real form of
$G$ is an antiholomorphic involutive automorphism $\sigma$ of $G$, so
$G(\R)=G(\C)^\sigma$ is a real Lie group. See Section \ref{s:real} for more details.
Let $\Gamma=\Gal(\C/\R)$
and write $H^i(\Gamma,G)$ for the Galois cohomology of $G$ (if $G$ is
nonabelian $i\le 1$).  If we want to specify how the nontrivial
element of $\Gamma$ acts we will write $H^i(\sigma,G)$.  The
equivalence (i.e. conjugacy) classes of real forms of $G$, which are
inner to $\sigma$ (see Section \ref{s:real}), are parametrized by
$H^1(\sigma,\Gad)$ where $\Gad$ is the adjoint group.

On the other hand, at least for $G$ connected, Cartan classified the
real forms of $G$ in terms of holomorphic involutions as follows. We
say a Cartan involution for $\sigma$ is a holomorphic involutive
automorphism $\theta$, commuting with $\sigma$, so that
$\sigma^c=\theta\sigma$ is a compact real form.
If $G$ is connected then $\theta$
exists, and is unique up to conjugacy by $G^\sigma$. 
Following Mostow we  prove a similar result in general. See Section \ref{s:real}.

Let $H^i(\Ztwo,G)$ be the group cohomology of $G$ where the nontrivial
element of $\Ztwo=\Z/2\Z$ acts by $\theta$.  As above we denote this
$H^i(\theta,G)$, and we refer to this as Cartan cohomology of $G$.
Conjugacy classes of involutions which are inner to $\theta$ are
parametrized by $H^1(\theta,\Gad)$.

Thus the equivalence of the two classifications of real forms amounts
to an isomorphism (for connected $G$) of the first Galois and Cartan cohomology spaces
$H^1(\sigma,\Gad)\simeq H^1(\theta,\Gad)$.  It is natural
to ask if the same isomorphism holds with $G$ in place of $\Gad$. 
For our applications it is helpful to know the result for disconnected groups as well. 

\begin{theorem}
\label{t:main}
Suppose $G$ is a complex, reductive algebraic group (not necessarily
connected), and $\sigma$ is a real form of $G$.
Let $\theta$ be a  Cartan involution for $\sigma$. 
Then there is a canonical isomorphism $H^1(\sigma,G)\simeq H^1(\theta,G)$. 
\end{theorem}

The interplay between the $\sigma$ and $\theta$ pictures plays a
fundamental role in the structure and representation theory of real
groups, going back at least to Harish Chandra's formulation of the
representation theory of $G(\R)$ in terms of $(\g,K)$-modules.  The
theorem is an aspect of this, and we give several applications.

Suppose $X$ is a homogeneous space for $G$, equipped with a real structure
$\sigma_X$ which is compatible with $\sigma_G$. Then the space of
$G(\R)$ orbits on $X(\R)$  can be understood in terms of the Galois
cohomology of the stabilizer of a point in $X$. Similar remarks apply
to computing $G^\theta$-orbits. 
Note that these stabilizers may be disconnected, even if $G$ is
connected.
See Proposition \ref{p:homogeneous}.

We use this principle to give simple proofs of  several well known results, including
the Kostant-Sekiguchi correspondence and Matsuki duality. 
Let $G(\C)$ be a connected complex
reductive group, with real form $\sigma$ and corresponding Cartan
involution $\theta$. Let $G(\R)=G(\C)^\sigma$, and $K(\C)=G(\C)^\theta$. Let $\g_0=\g^\sigma$
and $\P=\g^{-\theta}$. 
The Kostant-Sekiguchi correspondence is a bijection between the 
nilpotent
$G(\R)$-orbits on $\g_0$  and the  nilpotent $K(\C)$-orbits on $\P$.
Matsuki duality is a bijection between the $G(\R)$ and $K(\C)$ orbits
on the flag variety of $G$. 
See Propositions \ref{p:ks} and  \ref{p:matsukiduality}.

On the other hand Proposition \ref{p:X} applied to the space of Cartan
subgroups 
gives a simple proof of another result of Matsuki:
there is a bijection between $G(\R)$-conjugacy classes of
Cartan subgroups of $G(\R)$ and $K$-conjugacy classes of
$\theta$-stable Cartan subgroups of $G$ \cite{matsuki}. 
Also a well known result about
two versions of the rational Weyl group (Proposition \ref{p:W})
follows. 

If $G$ is connected Borovoi proved $H^1(\sigma,G)\simeq H^1(\sigma,H_f)/W_i$ where 
$H_f$ is a fundamental Cartan subgroup, and $W_i$ is a certain subgroup of the Weyl group \cite{borovoi}.
Essentially the same proof carries over to give 
$H^1(\theta,G)\simeq H^1(\theta,H_f)/W_i$. We prove this as a consequence 
of Theorem \ref{t:main}
(Proposition \ref{p:borovoi}).

Let $Z$ be the center of $G$ and let $\Ztor$ be its torsion subgroup.
Associated to  a real form 
$\sigma$ is its central invariant, denoted
$\zinv(\sigma)\in \Ztor^\sigma/(1+\sigma)\Ztor$.  
The formulation of a precise version of the Langlands correspondence 
requires the notion of strong real form. 
See Section \ref{s:strong} for this definition, and for 
the notion of central invariant of a strong real form, which is an element of $\Ztor^\sigma$. 

\begin{theorem}[Proposition \ref{p:H1}]
\label{t:H1}
Suppose $\sigma$ is a real form of $G$.
Choose a representative $z\in
\Ztor^{\sigma}$ of $\zinv(\sigma)\in \Ztor^{\sigma}/(1+\sigma)\Ztor$. Then there is a bijection
$$
H^1(\Gamma,G)\overset{1-1}\longleftrightarrow\text{the set of
  strong real forms with central invariant }z.
$$
\end{theorem}

This bijection is useful in both directions. On the one hand it is not
difficult to compute the right hand side, thereby computing
$H^1(\sigma,G)$.
Over a p-adic field $H^1(\sigma,G)=1$ if $G$ is simply connected.
Over $\R$ this is not the case, and we use Theorem \ref{t:main} to compute
$H^1(\sigma,G)$ for all such groups.
See Section \ref{s:real} and  the tables in Section
\ref{s:tables}. We used 
the Atlas of Lie Groups and
Representations software for some of these calculations.
See \cite{borovoi_timashev} for another approach. 

On the other hand the notion of strong real form is
important in formulating  a precise version of the local Langlands
conjecture. In that context it would be more natural if 
strong real forms were described in terms of classical Galois cohomology.
The Theorem provides such an interpretation. See Corollary
\ref{c:interpret}.

The authors would like to thank Michael Rapoport for asking about the
interpretation of strong real forms in terms of Galois cohomology, and
apologize it took so long to get back to him. 
We are also grateful to
Tasho Kaletha for several helpful discussions during the writing of this paper 
and of \cite{kaletha_rigid}, and Skip Garibaldi for a discussion of 
the Galois cohomology of the spin groups.

\sec{Preliminaries on Group Cohomology}
\label{s:prelim}

See \cite{serre_galois} for an overview of group cohomology.

For now suppose $\tau$ is an involutive automorphism of an abstract group $G$.
Define $H^i(\Ztwo,G)$ to be the group cohomology space where
the nontrivial element of $\Ztwo$ acts by $\tau$.  
We will also denote this group $H^i(\tau,G)$ 
\footnote[1]{There is a small notational issue here. If $\tau=1$ (the identity automorphisms of $G$), $H^1(1,G)$ denotes the group 
$H^1(\Ztwo,G)$ with $\Ztwo$ acting trivially.}.
If $G$ is abelian these are groups and are defined for all $i\ge 0$. 
Otherwise these are pointed sets, and defined  only  for $i=0,1$.
Let 
$$
Z^1(\tau,G)=G^{-\tau}=\{g\in G\mid g\tau(g)=1\}.
$$
Then we have
the standard identifications
$$
H^0(\tau,G)=G^\tau,\,
H^1(\tau,G)=Z^1(\tau,G)/\{g\mapsto x g\tau(x\inv)\}
$$
where $G^{-\tau}=\{g\in G\mid g\tau(g)=1\}$.  For $g\in G^{-\tau}$ let $\cl(g)$ be 
the corresponding class in $H^1(\tau,G)$.

If $G$ is abelian we also have the Tate cohomology groups $\wh H^i(\tau,G)$ $(i\in\Z)$.
These satisfy
$$
\wh H^0(\tau,G)=G^\tau/(1+\tau)G,\quad \wh H^1(\tau,G)=H^1(\tau,G),
$$
and (since $\tau$ is cyclic), 
$\wh H^i(\tau,G)\simeq \wh H^{i+2}(\tau,G)$ for all $i$.

Suppose $1\rightarrow A\rightarrow B\rightarrow C\rightarrow 1$ is an exact
sequence of groups with an involutive automorphism $\tau$.
Then there is an exact sequence
\begin{equation}
\label{e:longexact}
1
\rightarrow H^0(\tau,A)
\rightarrow H^0(\tau,B)
\rightarrow H^0(\tau,C)
\rightarrow H^1(\tau,A)
\rightarrow H^1(\tau,B)
\rightarrow H^1(\tau,C)
\end{equation}
Furthermore if $A\subset Z(B)$ ($Z(*)$ denotes the center of a group) then there is one further step
$\rightarrow H^2(\tau,A) = A^{\tau} / (1+\tau)A$. 

We will need the following generalization of  $H^1(\tau,G)$.
\begin{definition}
Suppose  $\tau$ is an involutive automorphism of $G$, and $A$ is a subset of $Z(G)$. 
Define 
\begin{subequations}
\renewcommand{\theequation}{\theparentequation)(\alph{equation}}  
\label{e:H1Z}  
\begin{equation}
Z^1(\tau,G;A)=\{g\in G\mid g\tau(g)\in A\}
\end{equation}
and
\begin{equation}
H^1(\tau,G;A)=Z^1(\tau,G;A)/[g \sim tg\tau(t\inv)\,\,(t\in G)].
\end{equation}
\end{subequations}
These are pointed sets if $1 \in A$.
The map $g \mapsto g\tau(g)$ factors to a map from $H^1(\tau,G;A)$ to $A$.
\end{definition}
Taking $A=\{1\}$ gives  ordinary cohomology $H^1(\tau,G)$. Write $\cl(g)$ for the image of $g\in Z^1(\tau,G;A)$ in $H^1(\tau,G;A)$. 

We make  use of twisting in nonabelian cohomology 
\cite{serre_galois}*{Section III.4.5}. 
Let $Z=Z(G)$. For $g\in G$ let $\int(g)$ be the inner automorphism
$\int(g)(h)=ghg\inv$. 
Fix an involutive automorphism $\tau$ of $G$, and $z\in Z$. 
Note that $\int(g)\circ\tau$ is an involution if and only if 
$g\in Z^1(\tau,G;Z)$.

\begin{lemma}
\label{l:twist1}
Suppose $\tau'=\int(g)\circ\tau$ for some $g\in Z^1(\tau,G;Z)$. 
Let $w=g\tau(g)\in Z$. Then the map
$h\mapsto hg\inv$ induces an isomorphism
$$
H^1(\tau,G;z)\rightarrow H^1(\tau',G;zw\inv).
$$
If $H^1(\tau,Z)=1$, this isomorphism is independent of the choice of $g\in Z^1(\tau,G;w)$ satisfying $\tau'=\int(g)\circ\tau$. 

In particular $H^1(\tau,G)\simeq H^1(\tau',G)$ if $\tau'=\int(g)\circ\tau$, where $g\in Z^1(\tau,G)$, 
and this isomorphism is canonical if $H^1(\tau,Z)=1$.

Finally suppose $\tau'$ is conjugate to $\tau$ by an inner automorphism of $G$. Then 
$H^1(\tau,G)\simeq H^1(\tau',G)$, and this isomorphism is canonical if
$$\ker\left( H^1(\tau,Z) \rightarrow H^1(\tau, G) \right)=1.$$
\end{lemma}

We omit the elementary proof.

Write $[\tau]$ for the $G$-conjugacy class of $\tau$. 

\begin{definition}
Assume $\ker\left( H^1(\tau,Z) \rightarrow H^1(\tau, G) \right)=1$. Given a G-conjugacy class $[\tau]$ of involutive automorphisms of $G$,
define $H^1([\tau],G)=H^1(\tau,G)$. 
\end{definition}
This is well-defined by the Lemma.

\sec{Real Forms and Cartan involutions}
 \label{s:real}

In the rest of the paper, unless otherwise noted, $G$ will denote a
complex, reductive algebraic group.
Except in a few places we do not
assume $G$ is connected. Write $G^0$ for the identity component. 

We identify $G$ with its complex points $G(\C)$ and use these
interchangeably. We may view $G$ either as an algebraic group or as a
complex Lie group. The identity component of $G$ as an algebraic group
is the same as the topological identity component when viewed as a Lie
group, the component group $G/G^0$ is finite.

A real form of $G$ is a real algebraic group $H$ endowed with
an isomorphism $\phi : H_{\C} \simeq G$, where $H_{\C}$ denotes the
base change of $H$ from $\R$ to $\C$.  By an {\em algebraic, conjugate
linear, involutive automorphism} of $H_\C$ we mean
an algebraic, involutive automorphism of $H_{\C}$ (considered as a
scheme over $\R$) such that the induced morphism between rings of
polynomial functions on $H$ is conjugate linear, and compatible with
the morphisms defining the group structure on $H$. Naturally associated to a
real form $H$ is an algebraic, conjugate linear, involutive automorphism $\sigma_H$
of $H_\C$.  Transporting $\sigma_H$ to $G$ via $\phi$ this is
equivalent to having an algebraic, conjugate linear, involutive automorphism $\sigma$ of $G$.
Conversely, by Galois descent any such automorphism  of $G$ comes
from a real form $(H, \phi)$, which is unique up to unique isomorphism.
See \cite{neronmodels}*{\S 6.2, Example B and \S 6.5} for details in
a much more general situation.

It is convenient to work with a more elementary notion of real form,
using only the structure of $G$ as a complex Lie group.  Any algebraic,
conjugate linear, involutive automorphism of $G$ induces an antiholomorphic
involutive automorphism of $G$. In fact every antiholomorphic
automorphism arises this way:

\begin{lemma} \label{l:antiholconjlin}
Let $G$ be a complex reductive
algebraic group.  Then any antiholomorphic involutive
automorphism of $G$ is induced by a unique algebraic conjugate linear involutive automorphism
of $G(\C)$.
\end{lemma}
\begin{proof}
Fix a representation $\rho : G \rightarrow \GL(V)$, where $V$ is a
complex vector space of finite dimension, such that $\rho$ is a
closed immersion \cite{borelbook}*{Proposition 1.10}.
Suppose that $\varphi : G \rightarrow G$ is an antiholomorphic involutive automorphism.
Choose an arbitrary real structure on $V$, and let $\sigma_V$ denote complex conjugation $\GL(V) \rightarrow \GL(V)$ with respect to this real structure.
Then $\sigma_V \circ \rho \circ \varphi$ is a holomorphic representation of $G$, so it is algebraic and $\varphi$ is algebraic conjugate linear. 
\end{proof}
The Lemma justifies the following elementary definition of real forms.

\begin{definition}
\label{d:realform}
A real form of $G$ is an antiholomorphic involutive automorphism $\sigma$ of $G$.
Two real forms are equivalent if they are conjugate by an inner automorphism.
Write $[\sigma]$ for the equivalence class of $\sigma$. 

We say two real forms $\sigma_1,\sigma_2$ are inner to each other, or
in the same inner class, if $\sigma_1\sigma_2\inv$ is an inner
automorphism of $G$. This is well defined on the level of equivalence
classes.
\end{definition}
See Remark \ref{r:notserre} for a subtle point regarding this notion of equivalence.

If $\sigma$ is a real form of $G$, let $G(\R)=G^\sigma$ be the fixed points  of $\sigma$. This is a real
Lie group, with finitely many connected components.

We turn now to compact real forms and Cartan involutions.  If $G$ is
connected these results are well known. The general case is due to Mostow \cite{mostow}.

\begin{definition}
\label{d:compact}
A real form $\sigma$ of $G$ is said to be a \emph{compact real form}
if $G^\sigma$ is compact and meets every component of $G$.
\end{definition}

Mostow's definition \cite{mostow}*{Section 2} of compact real form
refers to the subgroup $G^\sigma$, rather than the automorphism $\sigma$.
Let us check that our definition is equivalent to this.

\begin{lemma}
For any complex reductive group $G$, the map $\sigma \mapsto G^{\sigma}$ is a bijection
between the set of compact real forms of $G$, in the sense of Definition \ref{d:compact},
to the set of compact real forms of $G$, in the sense of \cite{mostow}.
\end{lemma}
\begin{proof}
If $\sigma$ is any real form of $G$, then $\dim_{\R} G^{\sigma} = \dim_{\C} G$, by Hilbert's
Theorem 90 applied to the action of $\sigma$ on $\Lie(G)$.
Choose a faithful algebraic representation $\rho : G \hookrightarrow \GL(V)$.
If $K$ is any compact subgroup of $G$, then $V$ admits a hermitian form for which $\rho(K)$ is unitary.
In particular we see that $\Lie(K) \cap i \Lie(K) = 0$.
These two facts imply that for any compact real form $\sigma$ of $G$, $G^{\sigma}$ is a compact real form of $G$ in the sense of \cite{mostow}.

Let us now check that $\sigma \mapsto G^{\sigma}$ is injective.
The action of $\sigma$ on $G^0$ is determined by its action on $\Lie(G) = \Lie(G^{\sigma}) \oplus i \Lie(G^{\sigma})$.
Once $\sigma|_{G^0}$ is determined, $\sigma$ is determined by the requirement that it fixes $G^{\sigma}$ pointwise,
since $G^{\sigma}$ meets every connected component of $G$.

Finally we show that $\sigma\mapsto G^\sigma$ is surjective.  Suppose
$K$ is a compact real form of $G$ in the sense of \cite{mostow}.
Choose $\rho$ and a hermitian form on $V$ as above.  Choosing an
orthonormal basis for $V$, we can view $\rho$ as a closed embedding
$G \rightarrow \GL_n(\C)$ such that $\rho(K) \subset U(n)$.  Let
$\tau(g)={}^tg^{-1}$ $(g\in GL_n(\C))$.  Then $\rho(G^0)$ is stable under
$\tau$, since $\Lie(\rho(G)) = \Lie(\rho(K)) \oplus i \Lie(\rho(K))$,
and $d\tau$ fixes $\Lie(\rho(K)) \subset \mathfrak{u}(n)$ pointwise.
Furthermore $\rho(G)$ is stable under $\tau$ since $\tau$ fixes $\rho(K)$
pointwise, and $G=G^0K$.  Pull back $\tau$ to $G$ to define
$\sigma=\rho\inv\circ\tau\circ\rho$.  This is a compact real form of
$G$, and $K\subset G^\sigma$.  By the Cartan decomposition
\cite{mostow}*{Lemma 2.1} $G^{\sigma} \cap G^0 = K \cap G^0$, and this
implies $G^{\sigma} = K$.
\end{proof}

Using the Lemma we will refer to $\sigma$ or  $K=G^\sigma$ as a compact real form of $G$.

The \emph{Cartan decomposition} holds in our setting (see \cite{mostow}*{Lemma 2.1}).

\begin{lemma}[Mostow]
\label{l:cartandecomp}
Suppose $\sigma$ is a compact real form of $G$. 
Let $K = G^{\sigma}$ and $\P = \Lie(G)^{-\sigma} = i \Lie(K)$. Then 
the map $(k, X) \mapsto k \exp(X)$ is a diffeomorphism 
from $K \times \P$ onto $G$ considered as a real Lie
group. Furthermore $K$ is a maximal compact subgroup of $G$.
\end{lemma}

Although we will not use it, it is not difficult to check that the
complexification functor \cite{serre_galois}*{Section III.4.5}, from
the category of compact Lie groups to that of complex reductive groups
endowed with a compact real form, induces a bijection on the level of
isomorphism classes.

It is important to know the existence of compact real forms. See   \cite{mostow}*{Lemma 6.1}.

\begin{theorem}[Weyl, Chevalley, Mostow]
Every complex reductive group has a compact real form.
\end{theorem}

We turn next to uniquess of the compact form.
See \cite{mostow}*{Theorem 3.1}, and  \cite{hochschild}*{Ch. XV} for a proof which handles one case overlooked 
in \cite{mostow}. 

\begin{theorem}[Cartan, Hochschild, Mostow] 
\label{t:conjcpctsubgps}
Let $\sigma$ be a compact real form of a complex reductive group $G$, and set $K=G^\sigma$.
Let $L$ be a compact subgroup of $G$.
Then there exists $g \in G^0$ such that $gL g^{-1} \subset K$.
The compact real forms of $G$ are unique up to conjugation by $G^0$.
\end{theorem}

Fix a compact real form $K$ of $G$.   The center $Z(G^0)$ of
$G^0$ is  a normal subgroup of $G$. 
If $G$ is connected it is well known that $Z(G)=Z(K)A$ 
where $A=\exp(i\Lie(Z(G^0)))\subset \exp(\P)$ is a vector group. 
Therefore in general we have
\begin{subequations}
\renewcommand{\theequation}{\theparentequation)(\alph{equation}}  
\label{e:center}
\begin{equation}
Z(G^0)=Z(K^0)A.
\end{equation}
Since $G=KG^0$ we have 
(writing superscript for invariants): $Z(G^0)^K=Z(G^0)^G$, independent of the choice
of $K$.  Also $K/K^0\simeq G/G^0$ acts on $Z(G^0)$, normalizing $A$, and 
\begin{equation}
Z(G)\cap G^0=Z(G^0)^{G/G^0}=Z(K^0)^{K/K^0}A^{G/G^0}=(Z(K)\cap K^0)A^{G/G^0}
\end{equation}
\end{subequations}

\begin{lemma} \label{l:normcpct}
Suppose $K$ is a compact real form of $G$. 
Then the Cartan decomposition of $\Norm_G(K)$ is 
 $\Norm_G(K) =K A^{G/G^0}$.
\end{lemma}
\begin{proof}
Since $G=K\exp(\P)$, it suffices to show that $\Norm_G(K) \cap \exp(\P) = A^{G/G^0}$.
Let $X \in \P$ be such that $\exp(X)$ normalizes $K$.
For $k \in K$, there exists $k' \in K$ such that $\exp(X) k \exp(-X) = k'$.
This can be rewritten as
$$ k \exp(-X) = k' \exp(- \Ad(k')^{-1}(X)) $$
so by uniqueness of the Cartan decomposition, $k'=k$ and $\Ad(k)(X)=X$, so  $X$ is invariant under $K$.
The fact that $X$ is invariant under $K^0$ means that $X \in \Lie(A)$, and since $K$ meets every connected component of $G$, $X \in \Lie(A)^{G/G^0}$.
\end{proof}

\begin{lemma}
\label{l:subcpctrealform}
Let $\sigma$ be a compact real form of a real reductive group $G$.
Let $H$ be a $\sigma$-stable algebraic subgroup of $G$. 
Then $H$ is reductive and $\sigma|_H$ is a compact real form of $H$.
\end{lemma}
\begin{proof}
The algebraic group $H$ is clearly linear.
The unipotent radical $U$ of $H$ is stable under $\sigma$ and connected, and so $U^{\sigma}$ is Zariski-dense in $U$.
Any unipotent element of $G^{\sigma}$ is trivial, thus $U=\{1\}$ and $H$ is reductive.
Clearly $H^{\sigma}$ is compact, and we are left to show that $H^{\sigma}$ meets every connected component of $H$.
For $h \in H$ write $h = k \exp (X)$ where $k \in G^{\sigma}$ and $X \in \P$.
Then $\exp(2X) = \sigma(h)^{-1} h \in H$, and thus $\exp(2n X) \in H$ for all $n \in \Z$.
Since $H$ is Zariski-closed in $G$ this implies $\exp(t X) \in H$ for all $t \in \C$,
which implies $X \in \h^{-\sigma}$, $k\in H^\sigma$, and $H^\sigma$ meets every component of $H$.
This argument is classical.
\end{proof}

\begin{definition}
Suppose $\sigma$ is a real form of a complex reductive group $G$.  A \emph{Cartan involution} for
$\sigma$ is a holomorphic involutive automorphism $\theta$ of $G$,
commuting with $\sigma$, such that $\theta\sigma$ is a compact real form of $G$.
\end{definition}

By Lemma \ref{l:antiholconjlin} applied to $\sigma$ and $\theta \sigma$, any
Cartan involution is algebraic. In fact a simple variant of the proof of Lemma
\ref{l:antiholconjlin} shows directly that any holomorphic automorphism of a
complex reductive group is automatically algebraic.

\begin{theorem}
\label{t:Cartaninv}
Let $G$ be a complex reductive group, possibly disconnected.
\begin{enumerate}
\item Suppose $\sigma$ is a real form of $G$.
\begin{enumerate}
\item There exists a Cartan involution $\theta$ for $\sigma$, unique
  up to conjugation by an 
  inner automorphism from $(G^\sigma)^0$.
\item Suppose $(H,\theta_H)$ is a pair consisting of a $\sigma$-stable reductive 
  subgroup of $G$ and a Cartan involution $\theta_H$ for
  $\sigma|_H$. 
  Then there exists a Cartan involution $\theta$ for $G$ such that $\theta(H)=H$ and $\theta|_H = \theta_H$.
\end{enumerate}
\item Suppose $\theta$ is a holomorphic, involutive automorphism of $G$.
\begin{enumerate}
\item There is a real form $\sigma$ of $G$ such that $\theta$  is a
  Cartan involution for $\sigma$, unique up to conjugation by an inner
  automorphism   from   $(G^{\theta})^0$.
\item Suppose $(H,\sigma_H)$ is a pair consisting of a $\theta$-stable reductive
   subgroup of $G$ and a real form $\sigma_H$ such that
  $\theta|_H$ is a Cartan involution for $\sigma_H$. 
  Then there exists a real form $\sigma$ of $G$ such that $\sigma(H)=H$ and $\sigma|_H = \sigma_H$.
\end{enumerate}
\end{enumerate}
\end{theorem}

For applications to the classification of real forms and to
homogeneous spaces, the fact that the statement of Theorem
\ref{t:Cartaninv} is symmetric in $\sigma$ and $\theta$ is crucial.

We will deduce  (1) and (2) from the next Lemma, whose proof is
adapted from \cite{mostow}*{Theorem 4.1}.  

\begin{lemma}
\label{l:tau}
Suppose $\tau$ is an involutive automorphism of $G$, either
holomorphic or anti-holomorphic.

\begin{enumerate}
\item  There exists a compact real form $\sigma^c$ of $G$ which commutes with $\tau$. 
\item Suppose $H$ is a $\tau$-stable reductive subgroup of $G$, $\sigma^c_H$ is a compact real form of $H$,
and $\tau$ commutes with $\sigma^c_H$. 
Then we can find $\sigma^c$ satisfying (1) so that $\sigma^c$ restricted to $H$ equals $\sigma^c_H$. 
\end{enumerate}
\end{lemma}

\begin{proof}
Choose any compact real form $\sigma^c_1$ of $G$ and set $K_1=G^{\sigma^c_1}$, $\P_1=\Lie(G)^{-\sigma^c_1}$, and $P_1=\exp(\P_1)$.
Then $\tau(K_1)$ is another compact real form of $G$, so by Theorem \ref{t:conjcpctsubgps} there exists $g\in G^0$ so that 
\begin{subequations}
\renewcommand{\theequation}{\theparentequation)(\alph{equation}}  
\begin{equation}
\tau(K_1)=gK_1g\inv.
\end{equation}
Applying $\tau$ to both sides we see $\tau(g)g\in \Norm_G(K_1)$.
By Lemma \ref{l:normcpct} we can write
\begin{equation}
\label{e:taug}
\tau(g)g=ak\quad (a\in A^{G/G^0},k\in K_1).
\end{equation}

By (a) $g\inv\tau(K_1)g=K_1$, i.e.\ $\int(g^{-1})\circ \tau$ stabilizes
$K_1$.
Since this isomorphism is holomorphic or antiholomorphic and $\P_1 = i \Lie(K_1)$, this implies $g\inv\tau(P_1)g=P_1$. 
By the Cartan decomposition $G=K_1P_1$ we may assume $g\in P_1$, in which
case 
$g\inv \tau(g)g\in P_1$. Plugging in  (b) 
we conclude $g\inv ak\in P_1$, which by 
uniquess of the Cartan decomposition  implies $k=1$, so
\begin{equation}
\tau(g)g\in A^{G/G_1}.
\end{equation}
Set $a=\tau(g)g\in Z(G)$. 
Then $\tau(a)=g \tau(g) = g a g^{-1} = a$. 
After replacing $g$ with $ga^{-\frac12}$ we may assume $\tau(g)=g\inv$
(we are writing $\frac12$ for the square root in the vector group $P_1$).
We observe that $g\inv \tau(g^{\frac12}) g$ is an element of $P_1$ and its square equals
$g\inv \tau(g) g = g$, therefore $\tau(g^{\frac12}) = g^{\frac12}$.

Now let $\sigma^c=\int(g^{\frac12})\circ\sigma^c_1\circ\int(g^{-\frac12})$,
$K=G^{\sigma^c}=g^{\frac12}K_1g^{-\frac12}$,
and $\P=\Lie(G)^{-\sigma^c}$.
Then
$$
\begin{aligned}
\tau(K)&=\tau(g^{\frac12})\tau(K)\tau(g^{-\frac12})=g^{-\frac12} gKg\inv g^{\frac12}=K.
\end{aligned}
$$
This also implies $\tau(\P)=\P$, and $\tau$ commutes with $\sigma^c$, as one can check using the Cartan decomposition.

Now suppose we are given $(H,\sigma^c_H)$ as in (2), and set
$K_H=H^{\sigma^c_H}$.  In the first step of the preceding argument
choose $\sigma^c_1$ so that $K_H\subset K_1$ (then $K_H=K_1\cap H$ since $K_H$ is a
maximal compact subgroup of $H$).  Suppose $h\in K_H$.  Choosing $g\in P_1$ as above, recall
$(\int(g\inv)\circ\tau)(K_1)=K_1$, so 
let $k=g\inv\tau(h)g\in K_1$.
Since $\tau$
commutes with $\sigma^c_H$,
$\tau(K_H)=K_H\subset K_1$,
so $\tau(h)\in K_1$.
Write 
\begin{equation}
kg\inv=\tau(h\inv)\cdot\tau(h\inv)g\inv\tau(h).
\end{equation}
By uniqueness of the Cartan decomposition we conclude
$g\tau(h)=\tau(h)g$ for all $h\in K_H$.
Since $\tau$ is an automorphism of $K_H$ we see $gh=hg$ for all
$h \in K_H$.
Since $\int(K_H) \subset \int(K_1)$ acts on $P_1$, this implies that
$g^{\frac12}h = hg^{\frac12}$ for all $h \in K_H$ as well.
Define $\sigma^c,K$ and $P$ as before.
Then $K_H=K\cap H$ and $\sigma^c(h) = h$ for all $h \in K_H$.
Now $(\sigma^c)^{-1} \circ \sigma^c_H : H \rightarrow G$ is a holomorphic automorphism
which is the identity on $K_H$, thus it is the identity on $H$ (recall that $\Lie(H) = \Lie(K_H) \oplus i \Lie(K_H)$ and that $K_H$ meets every connected component of $H$).
\end{subequations}

\end{proof}

\begin{proof}[Proof of Theorem \ref{t:Cartaninv}]
For existence in (1)(a) apply the Lemma to $\tau=\sigma$ to construct 
a compact real form $\sigma^c$, commuting with $\sigma$, and set 
$\theta=\sigma\sigma^c$.
For (1)(b) apply Lemma \ref{l:tau}(2)
with $\tau=\sigma$, $\sigma^c_H=\sigma|_H\theta_H$ to 
construct $\sigma^c$, commuting with $\sigma$,  and let $\theta=\sigma\sigma^c$. 

We now prove the uniqueness statement in (1)(a).
Suppose $\theta,\theta_1$ commute with $\sigma$, and $\sigma^c=\sigma\theta$ 
and $\sigma_1^c=\sigma\theta_1$ are compact real forms. By 
Theorem \ref{t:conjcpctsubgps} there exists $g\in G^0$ so that 
$$
\sigma^c_1=\int(g)\circ\sigma^c\circ\int(g\inv)=\int(g\sigma^c(g\inv))\circ\sigma^c.
$$
Let $G=K\exp(\P)$ be the Cartan decomposition with respect to $\sigma^c$. 
Then 
we can take $g=\exp(X)$ for $X\in\P$, 
so $g\sigma^c(g\inv)=\exp(2X)$. 
Since $\sigma^c$ and $\sigma^c_1$ commute with $\sigma$, so does $\int(g\sigma^c(g\inv))=\int(\exp(2X))$,
so by \eqref{e:center}(b)
$$
\exp(2\sigma(X))\exp(-2X)\in Z(G)\cap G^0=(Z(K)\cap K^0)A^{G/G^0}.
$$
Applying the Cartan decomposition for $\sigma^c$ again we conclude
$$\exp(2\sigma(X))\exp(X)\in A^{G/G^0},$$
so $\sigma(X)-X\in A^{G/G^0}$. 
We are free to multiply $g$ by
an element of $Z(G) \cap G^0$, which contains $A^{G/G^0}$.  In
particular we can replace $X$ with  $X+(\sigma(X)-X)/2\in \P^{\sigma}$. 
Then $g\in \exp(\P^\sigma)\in (G^\sigma)^0$. 

The proof of (2) is similar. We apply Lemma \ref{l:tau} with $\tau=\theta$. 
For existence in (2)(a) apply part (1) of the Lemma to construct $\sigma^c$, commuting with $\theta$, and let $\sigma=\theta\sigma^c$. 
For (2)(b) apply part (2) of the Lemma with $\sigma^c_H=\sigma_H\theta|_H$ to construct $\sigma^c$, commuting with $\theta$, 
and let $\sigma=\theta\sigma^c$. 
We omit the proof of the conjugacy statement, which is similar to case (1)(a).
\end{proof}

Let $\Int(G)$ be the group of inner automorphisms of $G$, $\Aut(G)$ the (holomorphic) automorphisms, and set
$\Out(G)=\Aut(G)/\Int(G)$.
Let $\Int^0(G)$ be the subgroup of $\Int(G)$ consisting of automorphisms induced by elements of $G^0$, so that $\Int^0(G) \simeq G^0/(Z(G) \cap G^0)$.

\begin{corollary}
\label{c:bijection}
The correspondence between real forms and Cartan involutions induces a bijection between
\begin{subequations}
\renewcommand{\theequation}{\theparentequation)(\alph{equation}}  
\label{e:bijection}
\begin{equation}
\{\text{antiholomorphic involutive automorphisms of $G$}\}/\Int^0(G)
\end{equation}
and 
\begin{equation}
\{\text{holomorphic involutive automorphisms of $G$}\}/\Int^0(G).
\end{equation}
\end{subequations}
Both quotients are by the conjugation action of inner automorphisms coming from $G^0$. 
The same statement holds with $\Int^0(G)$ replaced by any group $\mathcal A$ satisfying 
$\Int^0(G)\subset \mathcal A\subset \Aut(G)$.
\end{corollary}

If $\Int^0(G)$ is replaced by $\Int(G)=\Gad$ then (a) 
is the set of equivalence classes of real forms of
$G$ (Definition \ref{d:realform}). 
We use this bijection to identify an equivalence class of
real forms with an equivalence class of Cartan involutions as in (b).

\sec{Borel-Serre's Theorem}
\label{s:borelserre}

In this section only $G$ denotes a \emph{real} Lie group. 
Since it requires no extra effort we work in 
the following generality.

\begin{definition}
We say a real Lie group  $G$ has a Cartan decomposition $(K,\P)$ if $K$ is a compact subgroup of $G$, $\P$ is a subspace of 
$\Lie(G)$ stable under $\Ad(K)$, and the map $(k,X)\mapsto k\exp(X)$ is a diffeomorphism from $K\times\P$ onto $G$. 
\end{definition}
It is easy to see that $K$ is necessarily a \emph{maximal} compact subgroup of $G$.

Recall we have a Cartan decomposition in the case that $G$ is the
group $H(\C)$ of complex points of a reductive group $H$ viewed as a
real group (Lemma \ref{l:cartandecomp}): for any compact real form
$\sigma^c$ of $H$, we have
$H = H^{\sigma^c} \exp(\Lie(H)^{-\sigma^c})$.  Although we will not
use this fact, it is easy to deduce that if $\sigma$ is a real form of
a complex reductive group $H$, then for any Cartan involution $\theta$
of $(H,\sigma)$, the Lie group $H(\R) = H^{\sigma}$ has a Cartan
decomposition $H(\R) = H(\R)^{\theta} \exp(\Lie(H(\R))^{-\theta})$.

More general real Lie groups $G$ admit a Cartan decomposition,
including many non-linear ones (for example the \emph{finite} covers
of $\SL_2(\R)$) or non-reductive ones (for example $G=H(\R)$ where $H$
is a real linear algebraic group).  On the other hand the universal
cover $\widetilde G$ of $\SL_2(\R)$ has a decomposition
$\widetilde G=L\exp(\P)$ where $L\simeq \R$ is the universal cover of
the circle, hence noncompact.  For a generalization of the Cartan
decomposition to any real Lie group having finitely many connected
components see \cite{hochschild}*{Ch.\ XV} or \cite{mostow}*{Theorem
  3.2}.

\begin{proposition}
\label{p:serregeneral}
Suppose $G$ is a real Lie group admitting a Cartan decomposition $(K,\P)$. 
Let $\tau$ be an involutive automorphism of $G$ which
preserves $K$ and $\P$. Let $Z_K = Z(G)\cap K$. 
The inclusion map $K\rightarrow G$ induces an isomorphism 
\begin{equation}
\label{e:serregeneral}
H^1(\tau,K;Z_K)\simeq H^1(\tau,G;Z_K)
\end{equation}
which respects the maps to $Z_K$.
\end{proposition}

The proof is adapted from \cite{borel serre}*{Th\'eor\`eme 6.8} (see
also \cite{serre_galois}*{Section III.4.5}). This specializes to 
Borel-Serre's Theorem (see \eqref{e:serre}). 

\begin{proof}
It is enough to prove this when $Z_K$ is replaced by $\{z\} \subset Z_K$ 
where $z$ is any single element of $Z_K$. 
The left hand side of \eqref{e:serregeneral} is 
\begin{subequations}
\renewcommand{\theequation}{\theparentequation)(\alph{equation}}  
\begin{equation}
\{k\in K\mid k\tau(k)=z\}/[k \sim tg\tau(t\inv)\quad (t\in K)]
\end{equation}
and the right hand side is 
\begin{equation}
\{g\in G\mid g\tau(g)=z\}/[g \sim tg\tau(t\inv)\quad (t\in G)].
\end{equation}
\end{subequations}
Consider the map $\phi$ from (a) to (b) induced by inclusion.

We first show that $\phi$ is surjective. Suppose $g\in G$
satisfies $g\tau(g)=z$. Let $P=\exp(\P)$, and write $g=kp$ with $k\in K,p\in P$.
Then $kp\tau(kp)=z$, which can be written
$$
k\tau(k)\cdot \tau(k\inv)p\tau(k)=z\cdot\tau(p\inv).
$$
By uniqueness of the Cartan decomposition we conclude $k\tau(k)=z$
and $\tau(k\inv)p\tau(k)=\tau(p\inv)$. 
The latter condition is equivalent to $kpk\inv=\tau(p\inv)$.
The set of $p \in P$ satisfying this condition is the exponential of the 
subspace $\{Y\in \P \mid \mathrm{Ad}(k)Y=-\tau(Y)\}$. Therefore 
$p=q^2$ for some $q\in P$ satisfying $kq=\tau(q\inv)k$. 
Then $g=kq^2=(kq)q=\tau(q\inv)kq$. Therefore $\phi$ takes $\cl(k)$ in (a) to $\cl(g)$ in (b).

We now show that $\phi$ is injective. 
Suppose $k,k'\in K$, $k\tau(k)=k'\tau(k')=z$, and $k'=tk\tau(t\inv)$
for some $t\in G$. Write $t\inv=xp$ with $x\in K,p\in P$.
Then $k'=p\inv x\inv k\tau(x)\tau(p)$, i.e.
$$
k'\cdot (k')\inv pk'=x\inv k\tau(x)\cdot \tau(p)
$$
By uniqueness of the Cartan decomposition we conclude $k'=x\inv
k\tau(x)$ with $x\in K$, i.e. $k$ and $k'$ are equivalent in (a).
\end{proof}

\begin{corollary}
\label{c:taumu}
Suppose $G$ is a real Lie group admitting a Cartan decomposition $(K,\P)$, and as before let $Z_K=Z(G) \cap K$.
Let $\tau, \mu$ be involutive automorphisms of $G$ which
preserve $K$ and $\P$, and assume that $\tau|_K=\mu|_K$. Then there are canonical 
isomorphisms
$$H^1(\tau,G;Z_K)\simeq H^1(\tau|_K, K, Z_K) \simeq H^1(\mu,G;Z_K)$$
compatible with the maps to $Z_K$. 
In particular there is a canonical isomorphism of pointed sets 
\begin{equation}
\label{e:taumu}
H^1(\tau,G)\simeq H^1(\mu,G).
\end{equation}
\end{corollary}

Now let $G$ be a complex reductive group, viewed as a real group.
Recall (Section \ref{s:real}) $G$ has a compact real form $\sigma^c$, and a Cartan decomposition 
$G=K\exp(\P)$.
Hence Proposition \ref{p:serregeneral} applies. 
Taking $\tau=\sigma^c$ and restricting to the fibres of $\{1\} \subset Z_K$ gives Borel-Serre's Theorem
\cite{borel serre}*{Th\'eor\`eme 6.8}, \cite{serre_galois}*{Section III.4.5}
\begin{equation}
\label{e:serre}
H^1(\sigma^c,K)\simeq H^1(\sigma^c,G).
\end{equation}
This admits the following natural generalization to arbitrary real forms.

\begin{corollary}
\label{c:main}
Suppose $G$ is  a complex, reductive algebraic group $G$, $\sigma$ is a real form of $G$, and $\theta$ is a Cartan involution for $\sigma$. 
Let $\sigma^c=\sigma\theta$.

There are canonical isomorphisms 
$$ H^1(\theta,G;Z^{\sigma^c}) \simeq H^1(\theta,G^{\sigma^c};Z^{\sigma^c}) \simeq H^1(\sigma,G;Z^{\sigma^c}). $$
In particular there is a canonical isomorphism of pointed sets:
$$ H^1(\theta,G)\simeq H^1(\sigma,G). $$
\end{corollary}
This follows from Corollary \ref{c:taumu} for the Cartan decomposition of $G$ induced by $\sigma^c$, using the fact that
$\sigma$ and $\theta$ agree on $K=G^{\sigma^c}$. 

\begin{exampleplain}
\label{ex:serre1}
If  $\theta$ is the identity, $H^1(\theta,G)$
is the set of conjugacy classes of involutions in $G$. 
If $G$ is connected this is 
in bijection with $H_2/W$, where $H$ is a Cartan subgroup,  $H_2$ is the group of involutions in $H$
and $W$ is the Weyl group (see Example \ref{ex:compact}).

On the other hand $H^1(\theta,G(\R))$
is the set of conjugacy classes of involutions in $G(\R)$,
i.e. $H(\R)_2/W$. Since $H(\R)$ is compact this is equal to $H_2/W$. So we recover 
\cite{serre_galois}*{Theorem 6.1}: $H^1(\sigma,G)\simeq H^1(\theta,G(\R))=H(\R)_2/W$.
\end{exampleplain}

\begin{exampleplain}
\label{ex:sl2}
Suppose $G=PSL(2,\C)$. This has two real forms, $PGL(2,\R)\simeq SO(2,1)$ and $SO(3)$.
Since $G$ is adjoint $|H^1(\sigma,G)|=2$ for either real form.

Now let $G=SL(2,\C)$. From  Example \ref{ex:serre1} if $G(\R)=SU(2)$ then $|H^1(\sigma,G)|=2$. 
On the other hand if $G(\R)=SL(2,\R)$ then 
it is well known that $H^1(\sigma,G)=1$.
Thus
 in contrast to the adjoint case, although $SL(2,\R)$ and $SU(2)$ are inner forms of each other, 
their cohomology is different. See Lemma \ref{l:clarify}.
\end{exampleplain}

\sec{Rational Orbits}
\label{s:rationalorbits}

We use the results of the previous section to study rational orbits of $G$-actions 
for real reductive groups.

Write 
\begin{subequations}
\renewcommand{\theequation}{\theparentequation)(\alph{equation}}  
\label{e:quad}
\begin{equation}
(G,\tau_G,X,\tau_X)
\end{equation}
to indicate the following situation, which occurs repeatedly. First of
all $G$ is an abstract group 
equipped with an involutive automorphism $\tau_G$, and $X$ is a set
equipped with an involutive automorphism $\tau_X$. 
Furthermore 
there is a  left action of $g:x\mapsto g\cdot x$ of $G$ on $X$.
We assume $(\tau_G,\tau_X)$ are 
{\em compatible}:
\begin{equation}
\label{e:compatible}
\tau_X(g\cdot X)=\tau_G(g)\cdot \tau_X(x)\quad (g\in G,x\in X).
\end{equation}
We will apply this with $G$ a complex group, $X$ a complex variety, 
and $\tau_G$ and  $\tau_X$ each  acting holomorphically or anti-holomorphically. 
\end{subequations}

When $X$ is a homogeneous space the following description of the set of
orbits for the action of $G^{\tau_G}$ on $X^{\tau_X}$ is well known.

\begin{lemma}
\label{l:orbitsinv}
In the setting of \eqref{e:quad} suppose $X$ is a homogenous space for $G$.
Assume that $X^{\tau_X} \neq \emptyset$,
 choose $x \in X^{\tau_X}$ and denote by $G^x$ the stabilizer of $x$.
Then we have a bijection
\begin{align*}
X^{\tau_X} / G^{\tau_G} & \rightarrow \ker \left( H^1(\tau_G, G^x) \rightarrow H^1(\tau_G, G) \right) \\
 g \cdot x & \mapsto \cl(g^{-1} \tau_G(g))
\end{align*}
\end{lemma}

If $\sigma_G$ is a compact real form of $G$ then $X^{\sigma_X}$ is a
homogeneous space for $G^{\sigma_G}$:

\begin{lemma}
\label{l:orbitscpct}
In the setting of \eqref{e:quad}, suppose 
$G$ is a complex reductive algebraic group, $X$ is a homogeneous space for $G$, 
and $\sigma_G$ is a compact real form of $G$.  Let $K=G^{\sigma_G}$. 

\begin{enumerate}
\item $K$ acts transitively on $X^{\sigma_X}$.
\item Suppose $H$ is a $\sigma_G$-stable subgroup of $G$, and $H=G^x$ for some $x\in X$. 
Assume $X^{\sigma_X}\ne\emptyset$. Then $H=G^y$ for some $y\in X^{\sigma_X}$.
\end{enumerate}
\end{lemma}
\begin{proof}

For (1), if $X^{\sigma_X}$ is empty there is nothing to prove, so choose $x\in X^{\sigma_X}$. 
By the previous lemma we have to show that 
\begin{equation}
\tag{a}\ker \left( H^1(\sigma_G,G^x) \rightarrow H^1(\sigma_G, G) \right)
\end{equation}
is trivial.
By Lemma \ref{l:subcpctrealform}  $\sigma_G$ restricts to a compact real form of $G^x$, so 
Proposition \ref{p:serregeneral} implies (a) is isomorphic to 
\begin{equation}
\tag{b}\ker \left( H^1(\sigma_G, (G^x)^{\sigma_G}) \rightarrow H^1(\sigma_G, G^{\sigma_G}) \right)
\end{equation}
which is clearly trivial, proving (1).

For (2) choose $x\in X^{\sigma_X}$. The set of  subgroups $H$ in (2) is identified with the set of $\sigma_G$-fixed elements of the homogeneous space $G/\Norm_G(G^x)$.
By (1) $G^{\sigma_G}$ acts transitively on this set.
Thus for any such $H$  there exists $g \in G^{\sigma_G}$ such that $H = g G^x g^{-1}$. Then 
$g\cdot x\in X^{\sigma_X}$ and $H=G^{g \cdot x}$.
\end{proof}

We next consider homogeneous spaces for noncompact groups.

\begin{proposition}
\label{p:homogeneous}
Suppose $G$ is a complex, reductive algebraic group, possibly disconnected,
acting transitively on a 
complex algebraic variety $X$.
Suppose we are given:
\begin{enumerate}
\begin{item}
a pair $(\sigma_G,\theta_G)$ consisting of a real form,  and a corresponding Cartan involution, of  $G$;
\item 
a pair $(\sigma_X,\theta_X)$ of commuting involutions of $X$, with
$\sigma_X$ antiholomorphic and $\theta_X$ holomorphic.
\end{item}
\end{enumerate}
Assume $(\sigma_G,\sigma_X)$ are compatible, 
and so are $(\theta_G,\theta_X)$ (see \eqref{e:compatible}).

Assume $X^{\sigma_X} \cap X^{\theta_X} \ne\emptyset$. Then the two natural maps
$$ X^{\sigma_X} / G^{\sigma_G} \leftarrow (X^{\sigma_X} \cap X^{\theta_X}) / (G^{\sigma_G} \cap G^{\theta_G})  \rightarrow X^{\theta_X} / G^{\theta_G}$$
are bijective.
\end{proposition}
\begin{proof}
Choose $x \in X^{\sigma_X} \cap X^{\theta_X}$.
Lemma \ref{l:orbitsinv} applied to $(G, \sigma_G, X, \sigma_X)$ provides an identification
$$
X^{\sigma_X} / G^{\sigma_G} \simeq \ker \left( H^1(\sigma_G, G^x) \rightarrow H^1(\sigma_G, G) \right)
$$
Similarly, Lemma \ref{l:orbitsinv} applied to $(G, \theta_G, X, \theta_X)$ gives
$$
X^{\theta_X} / G^{\theta_G} \simeq \ker \left( H^1(\theta_G, G^x) \rightarrow H^1(\theta_G, G) \right)
$$
Let $\sigma^c_G=\sigma_G\theta_G$. 
By Lemma \ref{l:orbitscpct}, $G^{\sigma^c_G}$ acts transitively on $X^{\sigma_X \theta_X}$, so that we can also apply Lemma \ref{l:orbitsinv} to $(G^{\sigma^c_G}, \sigma_G, X^{\sigma_X \theta_X}, \sigma_X)$:
$$
(X^{\sigma_X} \cap X^{\theta_X}) / (G^{\sigma_G} \cap G^{\theta_G}) \simeq \ker \left( H^1(\sigma_G, (G^x)^{\sigma^c_G}) \rightarrow H^1(\sigma_G, G^{\sigma^c_G}) \right).
$$
By Corollary \ref{c:main} we have the following commutative diagram:
$$
\xymatrix{
H^1(\sigma_G, G^x) \ar@{->}[d] \ar@{<-}[r]^{\simeq} & H^1(\sigma_G, (G^x)^{\sigma^c_G}) \ar@{->}[d] \ar@{->}[r]^{\simeq} & H^1(\theta_G, G^x) \ar@{->}[d] \\
H^1(\sigma_G, G) \ar@{<-}[r]^{\simeq} & H^1(\sigma_G, G^{\sigma^c_G}) \ar@{->}[r]^{\simeq} & H^1(\theta_G, G)
}
$$
Note that $\sigma_G$ and $\theta_G$ coincide on $G^{\sigma^c_G}$ so in the middle term 
we can replace $H^1(\sigma_G,*)$ with $H^1(\theta_G,*)$. 
This gives the two bijections of the Proposition. 

These bijections (which involve the choice of $x$) agree with those
of the Proposition (which are canonical). This comes down to:  if $g\in G^{\sigma^c_G}$ then 
$g\inv \sigma_G(g)=g\inv\theta_G(g)$. This completes the proof. 
\end{proof}

\begin{remarkplain}
In Proposition \ref{p:X}, the hypothesis $X^{\sigma}\cap X^\theta\ne\emptyset$ is necessary.
Consider for example $G=X=\C^\times$, with $G$ acting by multiplication,
and $\sigma_G(z)=1/\overline z,\sigma_X(z)=-1/\overline z,\theta_G(z)=\theta_X(z)=z$.
Then $X^{\sigma_X} = \emptyset$ but $X^{\theta_X} = X$.
\end{remarkplain}

To apply the result it would be good to know that $X^{\sigma_X}\ne\emptyset$ 
or $X^{\theta_X} \neq \emptyset$ implies that $X^{\sigma_X} \cap
X^{\theta_X} \neq \emptyset$. 
As the Remark shows, this isn't always the case, but it holds under a weak additional assumption.

\begin{lemma}
\label{l:nonempty}
In the setting of the Proposition, 
assume that $X^{\sigma_X\theta_X}\ne\emptyset$.
Then the following conditions are equivalent:
$X^{\sigma_X}\ne\emptyset$, $X^{\theta_X} \ne \emptyset$, and 
$X^{\sigma_X} \cap X^{\theta_X} \ne \emptyset$.
\end{lemma}
\begin{proof}
If $x\in X^{\sigma_X\theta_X}$ then  $G^x$ is $\sigma^c$-stable 
so $G^x$ is reductive by Lemma \ref{l:subcpctrealform}.
Since these groups are all conjugate this holds for all $x\in X$.

Let us now show that if 
$X^{\sigma_X} \neq \emptyset$  then
$X^{\sigma_X} \cap X^{\theta_X} \neq \emptyset$.  Fix
$x \in X^{\sigma_X}$. Then $G^x$ is a reductive group stable under
$\sigma_G$, and thus it admits a Cartan involution $\theta_x'$.  By
Theorem \ref{t:Cartaninv} it extends to a Cartan involution
$\theta_G'$ of $G$, and there exists $g \in G^{\sigma_G}$ such that
$\theta_G = \int(g) \circ \theta_G' \circ \int(g^{-1})$, so that
$g \cdot x \in X^{\sigma_X}$ has the property that $G^{g \cdot x}$ is
$\theta_G$-stable.  In other words, after replacing $x$ by
$g \cdot x$, we may assume $G^x$ is $\sigma^c$-stable,
and $\sigma_c|_{G^x}$ is a compact real form of $G^x$.
By Lemma \ref{l:orbitscpct} we can find 
$y\in X^{\sigma_X\theta_X}$ so that $G^y=G^x$. 

Let $N_y=\Norm_G(G^y)$, and set $M_y=N_y/G^y$, 
By \cite{springer_book}*{Proposition 5.5.10} $M_y$ is a linear
algebraic group.
Both $N_y$ and $M_y$ are $\sigma^c$-stable,  and therefore reductive by
Lemma \ref{l:subcpctrealform} again.

\begin{subequations}
\renewcommand{\theequation}{\theparentequation)(\alph{equation}}  
Since  $G^{\sigma_X(y)} = \sigma_G(G^y) = G^y$ 
there exists unique $m\in M_y$ such that
\begin{equation}
\sigma_X(y)=m\cdot y.
\end{equation}
Similarly since $G^x=G^y$  there exists unique $n\in M_y$ such that 
\begin{equation}
x=n\cdot y
\end{equation}
\end{subequations}

Since 
$\sigma_X\theta_X$ fixes both $y$ and $\sigma_X(y)$,
applying this to both sides of (a) gives
$\sigma_X(y)=\sigma^c(m)\cdot y$, and comparing this with (a) gives  $m\in (M_y)^{\sigma^c}$. 
On the other hand applying $\sigma_X$ to both sides of (a) gives
$y=\sigma_G(m)\cdot\sigma_X(y)=\sigma_G(m)m\cdot y$, so 
$\sigma_G(m)m=1$. 
Finally apply $\sigma_X$ to both sides of (b) to give
$\sigma_X(x)=\sigma_G(n)\cdot\sigma_X(y)$. 
Using $\sigma_X(x)=x$ and (a) gives $x=\sigma_G(n)m\cdot y$, and
comparing this with (b) gives $\sigma_G(n)\inv n=m$. 

These three facts imply that
$m$ defines an element of
$$ \ker\left(H^1(\sigma_G, (M_y)^{\sigma^c}) \rightarrow H^1(\sigma_G, M_y) \right). $$
By Corollary \ref{c:main} this kernel is trivial, so  there exists $u \in (M_y)^{\sigma^c}$ such that $h = \sigma_G(u)^{-1} u$.
Then $u \cdot y \in X^{\sigma_X\theta_X} \cap X^{\sigma_X} = X^{\sigma_X} \cap X^{\theta_X}$.

A similar argument, substituting $\theta$ for $\sigma$, shows that $X^{\theta_X} \neq \emptyset$ implies that $X^{\sigma_X} \cap X^{\theta_X} \neq \emptyset$.
\end{proof}

We can now formulate our result in its most useful form.

\begin{proposition}
\label{p:X}
Suppose $G$ is a complex, reductive algebraic group, possibly disconnected, 
and $X$ is a complex algebraic variety, equipped with an action of $G$. 
Suppose we are given:
\begin{enumerate}
\begin{item}
a pair $(\sigma_G,\theta_G)$ consisting of a real form and a corresponding Cartan involution of $G$.  
\item 
a pair $(\sigma_X,\theta_X)$ of commuting involutions,  with
$\sigma_X$ antiholomorphic and $\theta_X$ holomorphic.
\end{item}
\end{enumerate}
Assume $(\sigma_G,\sigma_X)$ are compatible, 
as are $(\theta_G,\theta_X)$ \eqref{e:compatible}.

Assume that for all $x\in X^{\sigma_X} \cup X^{\theta_X}$ the $G$-orbit of $x$
intersects $X^{\sigma_X\theta_X}$. 
Then the two natural maps
$$ X^{\sigma_X} / G^{\sigma_G} \leftarrow (X^{\sigma_X} \cap X^{\theta_X}) / (G^{\sigma_G} \cap G^{\theta_G})  \rightarrow X^{\theta_X} / G^{\theta_G}$$
are bijective.
\end{proposition}
\begin{proof}
It is enough to prove this  with $X$ replaced by the $G$-orbit $G\cdot
x$ of any  $x\in X^{\sigma_X} \cup X^{\theta_X}$.
By Lemma \ref{l:nonempty} we can apply
Proposition \ref{p:X} to $G\cdot x$, which gives the conclusion.
\end{proof}

\sec{Applications}
\label{s:applications}

Throughout this section we fix a 
\emph{connected} complex reductive group $G$, a real form $\sigma$ of
$G$, and a corresponding Cartan involution $\theta$. Set $G(\R)=G^\sigma$
and $K=G^\theta$.

\subsec{Kostant-Sekiguchi correspondence}
\label{s:ks}

Let $\g=\Lie(G)$. 
The Jacobson-Morozov theorem (see \cite{bourbakiLie78}*{ch. VIII, \S 11}) gives a bijection  between the nilpotent
orbits of $G$ on $\g$ and $G$-conjugacy classes of homomorphisms from
$\sl(2,\C)$ to $\g$:
\begin{subequations}
\renewcommand{\theequation}{\theparentequation)(\alph{equation}}  
\begin{equation}
\{\phi:\sl(2,\C)\rightarrow \g\}/G.
\end{equation}

Let $\g_0=\Lie(G(\R))=\g^\sigma$. 
Then the same result applies to $G(\R)$, and gives a bijection
between the $G(\R)$ conjugacy classes of nilpotent elements of $\g_0$
and
\begin{equation}
\{\phi:\sl(2,\R)\rightarrow \g_0\}/G(\R).
\end{equation}
Equivalently if $\sigma_0$ denotes   complex conjugation on $\sl(2,\C)$ with
respect to $\sl(2,\R)$, then (b) can be replaced with 
\begin{equation}
\{\phi:\sl(2,\C)\rightarrow\g\mid \phi(\sigma_0 X)=\sigma(\phi(X))\}/G(\R).
\end{equation}

Now write $\g=\k\oplus\P$ where 
$\k=\g^\theta=\Lie(K)$ and $\P=\g^{-\theta}$. 
For $X\in\sl(2,\C)$ define  $\theta_0(X)=-{}^tX$; this is a  Cartan involution for $\sigma_0$.
Kostant and Rallis \cite{kostant_rallis} showed that the nilpotent $K$-orbits on
$\P$ are in bijection with
\begin{equation}
\{\phi:\sl(2,\C)\rightarrow\g\mid \phi(\theta_0(X))=\theta(\phi(X))\}/K.
\end{equation}
\end{subequations}

The Kostant-Sekiguchi correspondence is a bijection between the nilpotent orbits of $G(\R)$ on $\g_0$ and the nilpotent
$K$-orbits on $\P$ \cite{sekiguchi_correspondence}.

Let $X$ be the set of morphisms
 $\sl(2,\C)\rightarrow \g$.
This has a natural structure of complex algebraic variety.
Define an antiholomorphic involution $\sigma_X$ of $X$ by
\begin{subequations}
\renewcommand{\theequation}{\theparentequation)(\alph{equation}}  
\begin{equation}
\sigma_X(\psi)(A)=\sigma(\psi(\sigma_0(A)))\quad (A\in\sl(2,\C),\psi\in X).
\end{equation}
\end{subequations}
Also define a holomorphic involution $\theta_X$ by
\begin{equation}
\theta_X(\psi)(A)=\theta(\psi(\theta_0(A)))\quad (A\in\sl(2,\C),\psi\in X).
\end{equation}
It is straightforward to check that 
($\sigma_G,\sigma_X)$ and $(\theta_G,\theta_X)$ are compatible.

\begin{lemma}
\label{l:KS}
Every orbit of $G$ on $X$ contains a $\sigma_X \theta_X$-invariant point.
In particular, $\sigma_X \theta_X$ acts trivially on $X/G$, and an orbit of $G$ on $X$ is $\sigma_X$-stable if and only if it is $\theta_X$-stable.
\end{lemma}
\begin{proof}
We need to show that for any morphism $\phi : \sl(2,\C)\rightarrow \g$, there exists $g \in G$ such that the morphism $\mathrm{Ad}(g) \circ \phi$ is $\sigma \theta$-equivariant.
Any such $\phi$ integrates to an algebraic morphism $\psi : \SL_2(\C) \rightarrow G^0$.
Let $SU(2)=SL(2,\C)^{\sigma_0\theta_0}$, with Lie algebra $\su(2)$.
Since $SU(2)$ is compact, so is its image in $G^0$, so 
by Theorem \ref{t:conjcpctsubgps} there exists $g \in G^0$ such that $g\psi(\SU(2))g\inv \subset (G^0)^{\sigma \theta}$.
Since $\sl(2,\C) = \su(2)\otimes_\R(\C)$
this implies that $\mathrm{Ad}(g) \cdot \phi$ is $\sigma \theta$-equivariant.
\end{proof}

The  Kostant-Sekiguchi correspondence is now an immediate consequence of Proposition \ref{p:X}.

\begin{proposition}
\label{p:ks}
For any nilpotent orbit $\O$ of $G$ on $\g$, there is a canonical bijection between $(\O\cap\g_0) / G(\R)$ and $(\O\cap \P) / K$.   
\end{proposition}
\begin{proof}
Let $\phi:\sl(2,\C)\rightarrow \g$ be a morphism corresponding to an element of $\O$ as in
(a). Let $Y \subset X$ be the $G$-orbit of $\phi$, which only depends on $\O$ and not on the choice of a particular morphism.
By Lemma \ref{l:KS}, $Y$ is $\sigma_X$-stable if and only if it is $\theta_X$-stable.
If it it not the case, both quotient sets are empty.

If it is the case we can apply Proposition \ref{p:X} to $X$, and by the Jacobson-Morozov theorem over $\R$ and the result of Kostant and Rallis recalled above, we obtain:
$$
(\O\cap\g_0) / G(\R) \simeq X^{\sigma_X}/G^{\sigma} \simeq
X^{\theta_X} / G^{\theta} \simeq (\O\cap \P) / K.
$$
\end{proof}
\begin{remarkplain}
The set of orbits $(X^{\sigma_X} \cap X^{\theta_X}) / (G^{\sigma} \cap G^{\theta})$ that appears as a middle term in Proposition \ref{p:X}, that is the set of $K(\R)$-conjugacy classes of morphism $\sl(2,\C)\rightarrow \g$ equivariant under $\sigma$ and $\theta$, does \emph{not} have an obvious link to nilpotent orbits, since $\P_0$ has no non-zero nilpotent elements.
\end{remarkplain}
\subsec{Matsuki Duality}
Matsuki duality is a bijection between the $G(\R)$ and $K$  orbits
on the space $\B$ of Borel subgroups of $G$ \cite{matsuki}. 

Unlike in the case of Kostant-Sekiguchi duality, $G(\R)$ and $K$
are acting on the same space $\B$. So to  derive this from Proposition \ref{p:X} we 
need to find $(X,\sigma_X,\theta_X)$ so that $X^{\sigma_X}\simeq
X^{\theta_X}\simeq\B$. This holds if we take  $X=\B\times\B$, and define
$\sigma_X(B_1,B_2)=(\sigma(B_1),\sigma(B_2))$, 
$\theta_X(B_1,B_2)=(\theta(B_1),\theta(B_2))$.
However with this definition the condition $X^{\sigma_X}\cap X^{\theta_X}\ne\emptyset$ of
Proposition \ref{p:X} does not hold. Also note that the stabilizer of
a point in $\B$ is the intersection of two Borel subgroups, which is
typically not reductive. Instead we use a variant of $X$.

Write $\sigma_G=\sigma,\theta_G=\theta$.

\begin{definition}
Let 
\begin{equation}
X=
\{(B_1,B_2,T)\mid 
B_1,B_2\in\B, 
T\subset B_1\cap B_2\text{ is a Cartan subgroup}\}
\end{equation}
Let $G$ act on 
$X$ by conjugation on each factor. 
Define involutive automorphisms $\sigma_X$ and $\theta_X$ of $X$ as
follows:
\begin{equation}
\sigma_X(B_1,B_2,T)=(\sigma_G(B_2\opp),\sigma_G(B_1\opp),\sigma_G(T))
\end{equation}
where $\opp$ denotes the opposite Borel with respect to $T$, and
\begin{equation}
\theta_X(B_1,B_2,T)=(\theta_G(B_2),\theta_G(B_1),\theta_G(T)).
\end{equation}
\end{definition}

Thanks to the Bruhat decomposition \cite{borelbook}*{\S 14.12}, for any $(B_1,
B_2) \in \B \times \B$ the algebraic subgroup $B_1 \cap B_2$ of $G$ is connected
and solvable and contains a maximal torus of $G$.
In particular the natural map $X \rightarrow \B \times \B$ is surjective.

\begin{lemma}
The conditions of Proposition \ref{p:X} hold. 
\end{lemma}

\begin{proof}
The fact that $\sigma_X,\theta_X$ commute, and the facts that 
$(\sigma_G,\sigma_X)$  and $(\theta_G,\theta_X)$ are compatible is
immediate.
Let us check that each $G$-orbit in $X$ contains a
$\sigma_X\theta_X$-fixed point.
Let $(B_1,B_2,T) \in X$.
Since the real reductive group $(G, \sigma_G \theta_G)$ has a maximal torus
defined over $\R$ \cite{borelbook}*{Theorem 18.2}, up to conjugating by an element of $G$ we can assume
that $T$ is $\sigma_G \theta_G$-stable. Since $(T, \sigma_G \theta_G)$ is anisotropic
we have $\sigma_G \theta_G(B_i) = B_i\opp$ for $i \in \{1,2\}$, and $(B_1,B_2,T)$ is
automatically fixed by $\sigma_X \theta_X$.
\end{proof}

Proposition \ref{p:X} now applies to give a bijection
\begin{equation}
\label{e:matsuki}
X/G(\R)\longleftrightarrow X/K.
\end{equation}

\begin{lemma}
Consider the projection $p$ on the first factor, taking $X$ to $\B$.
\begin{enumerate}
\item  $p$ restricted to $X^{\sigma_X}$ is equivariant with respect to
  $G(\R)$ and induces a bijection $X^{\sigma_X}/G(\R)\simeq \B/G(\R)$.
\item  $p$ restricted to $X^{\theta_X}$ is equivariant with respect to
  $K$ and induces a bijection $X^{\theta_X}/K\simeq \B/K$.
\end{enumerate}
\end{lemma}

\begin{proof}
The fact that $p$ is $G$-equivariant, and $p|_{X^{\sigma_X}}$ is
$G(\R)$-equivariant, are immediate.
Let $B$ be a Borel subgroup of $G$.
Then $B \cap \sigma_G(B)$ is an algebraic subgroup of $G$ defined over $\R$, and
so it contains a maximal torus $T$ which is defined over $\R$.
The Bruhat decomposition implies that $T$ is also a maximal torus of $G$.
This shows that $B \in p(X^{\sigma_X})$.

Moreover the unipotent radical $U$ of $B$ acts transitively on the set of maximal tori of $B$
\cite{borelbook}*{Theorem 10.6}, and since $G$ is reductive this action is also
free.
Therefore $U^{\sigma_G}$ acts simply transitively on the set of
$\sigma_G$-stable maximal tori in $B$.
This implies that $p$ induces a bijection $X^{\sigma_X}/G(\R)\simeq \B/G(\R)$.

The proof of (2) is similar, except for the fact that $B \cap \theta_G(B)$
contains a maximal torus which is $\theta_G$-stable, which follows from
\cite{steinberg}*{7.6} applied to $\theta_G$ acting on $B \cap \theta_G(B)$.
\end{proof}

Together with \eqref{e:matsuki} this proves:
\begin{proposition}
\label{p:matsukiduality}
There is a canonical bijection $\B/G(\R)\leftrightarrow \B/K$.
\end{proposition}

\subsec{Weyl groups and conjugacy of Cartan subgroups}
\label{s:cartans}

We next give short proofs of two well known facts about Weyl groups and
conjugacy of Cartan subgroups.

Let $X$ be the set of Cartan subgroups of $G$. This is a homogeneous
space for the conjugation action of $G$, with 
$\sigma_X,\theta_X$ coming from $\sigma$ and $\theta$. 
It is well known that $G$ has a $\sigma$-stable Cartan
subgroup, that is $X^{\sigma_X} \neq \emptyset$.
This also applies to $G$ equipped with its real form $\sigma \theta$, so that $X^{\sigma_X\theta_X} \neq \emptyset$.

Matsuki's result on Cartan subgroups (\cite{matsuki},\cite{av1}*{Proposition 6.18}) now follows from Proposition \ref{p:X}.

\begin{proposition}
\label{p:matsuki}
There are canonical bijections between
\begin{itemize}
\item $G(\R)$-conjugacy classes of $\sigma$-stable 
Cartan subgroups of $G$,
\item $K(\R)$-conjugacy classes of $\sigma$- and $\theta$-stable Cartan subgroups of $G$,
\item $K$-conjugacy classes of $\theta$-stable Cartan subgroups of $G$.
\end{itemize}
\end{proposition}

In particular we recover the fact that $G$ admits a $\theta$-stable Cartan subgroup $H$ in every $G(\R)$-conjugacy class of $\sigma$-stable Cartan subgroups.

Next, we  recover the following description of the \emph{real} or \emph{rational} Weyl group of $H$.
See also \cite{warner_I}*{Proposition 1.4.2.1}, \cite{vogan_green}*{Definition 0.2.6}.

\begin{proposition}
\label{p:W}
Let $H$ be a Cartan subgroup of $G$ which is both $\sigma$ and
$\theta$ stable.
Then the two natural morphisms
$$
\Norm_{G(\R)}(H(\R))/H(\R) \leftarrow \Norm_{K(\R)}(H(\R))/H(\R)^{\theta} \rightarrow \Norm_{K}(H)/(H\cap K)
$$
are isomorphisms.
\end{proposition}
\begin{proof}
Naturally $\sigma$ and $\theta$ act on $N=\Norm_G(H)$ and on the Weyl group $W=N/H$.
Note that the three quotients in the Proposition are $N^{\sigma}/H^{\sigma}$, (resp.\ $(N^{\sigma} \cap N^{\theta}) / (H^{\sigma} \cap H^{\theta})$, $N^{\theta}/H^{\theta}$), and thus are naturally subgroups of $W^{\sigma}$ (resp.\ $W^{\sigma} \cap W^{\theta}$, $W^{\theta}$).
Denote by $\pi$ the canonical surjective morphism $N \rightarrow W$.
By Lemma \ref{l:subcpctrealform}, $N^{\sigma \theta}$ meets every connected component of $N$.
Since $H=N^0$, this means that $N^{\sigma \theta}$ maps surjectively to $W$.
In particular, $\sigma \theta$ acts trivially on $W$, and so $W^{\sigma} = W^{\theta}$.

For $w \in W^{\sigma}$ let $N_w=\pi^{-1}(\{w\})$. This is a torsor under $H$ that contains a $\sigma \theta$-invariant point.
By Corollary \ref{c:main} the following conditions are equivalent: 
$(N_w)^{\sigma}\ne\emptyset$,$(N_w)^{\theta}\ne\emptyset$, and $(N_w)^{\sigma\theta}\ne\emptyset$.
\end{proof}

\medskip

\sec{Relation with Cohomology of Cartan subgroups}
\label{s:torus}

We continue to assume $G$ is a connected complex reductive group. 
Suppose $\sigma$ is a real form of $G$, 
and $\theta$ is a Cartan involution for $\sigma$. 

We say a $\sigma$-stable Cartan subgroup $H_f$ of $G$ is fundamental
if $H_f(\R)$ is of minimal split rank.  Borovoi computes
$H^1(\sigma,G)$ in terms of $H^1(\sigma,H_f)$ as follows.
Before stating his result we make a few remarks about Weyl groups.

\begin{lemma}
\label{l:H1N}
Suppose $H$ is a $\sigma$-stable Cartan subgroup.
There is an action of $W^\sigma$ on $H^1(\sigma,H)$ defined as
follows. 
Suppose $w\in W^\sigma$ and  $h\in H^{-\sigma}$. 
Choose $n\in N$ mapping to $w$.
Then the action of $w$ on $H^1(\sigma,H)$ is 
$w:\cl(h)\rightarrow \cl(nh\sigma(n\inv))$; this is well defined, independent of the choices involved.

The image of $H^1(\sigma,H)$ in $H^1(\sigma,N)$ is isomorphic to  $H^1(\sigma,H)/W^\sigma$.
\end{lemma}

This is immediate. See \cite{serre_galois}*{I.5.5, Corollary 1}.

Suppose a Cartan subgroup  $H$  is $\sigma$-stable. Then $\sigma$ acts on the roots of $H$
in $G$. We say a root $\alpha$ of $H$ in $G$ is imaginary, real, or complex if 
$\sigma(\alpha)=-\alpha$, $\sigma(\alpha)=\alpha$, or
$\sigma(\alpha)\ne\pm\alpha$, respectively. 
The set of imaginary roots is a root system.
Let $W_i$ denote its Weyl group.

\begin{lemma}
\label{l:Wi}
$H^1(\sigma,H)/W^\sigma= H^1(\sigma,H)/W_i$.
\end{lemma}
\begin{proof}
Write $W^\sigma=(W_C)^\sigma\ltimes[W_i\times W_r]$   
as in \cite{ic4}*{Proposition 4.16}. Here $W_r$ is Weyl group of the real
roots, and 
$(W_C)^\sigma$ is a certain Weyl group, generated by terms of the form
$s_\alpha s_{\sigma\alpha}$ where $\alpha,\sigma\alpha$ are
orthogonal. 
It is easy to see that $W_r$ acts trivially on $H^1(\sigma,H)$, 
and $(W_C)^\sigma$ does as well \cite{algorithms}*{Proposition 12.16}.
\end{proof}

\begin{proposition}[Borovoi \cite{borovoi}]
Suppose $H_f$ is a fundamental $\sigma$-stable Cartan subgroup. 
The natural map 
$H^1(\sigma,H_f)\rightarrow H^1(\sigma,G)$ induces an isomorphism 
$H^1(\sigma,H_f)/W_i\simeq H^1(\sigma,G)$.  
\end{proposition}

The Theorem in \cite{borovoi} is stated in terms of $W^\sigma$, so we
have used the preceding Lemma to replace this with $W_i$.

\begin{remarkplain}
Borovoi has pointed out that we can  replace $W_i$ with
another group, which is much smaller in the unequal rank case.
Fix a pinning $(H_f,B,\{X_\alpha\})$. The inner class of $\sigma$
corresponds to an involution $\delta\in\Aut(G)$  which preserves the pinning
(see Section \ref{s:real}). Thus $\delta$ defines an involution of the
simple roots, which is trivial if and only if the derived group is
equal rank. 

Let $W_0$ be the Weyl group generated by the $\delta$-fixed simple
roots. For example in type $A_n$, if $\delta$ is nontrivial, then 
$W_0$ is trivial if $n$ is even, or $\Ztwo$ if $n$ is odd. 
Borovoi proves that $H^1(\sigma,H_f)/W_i\simeq H^1(\sigma,H_f)/W_0$. 
\end{remarkplain}

\begin{proposition}
\label{p:borovoi}
There is a canonical isomorphism
$\phi:H^1(\theta,G)\simeq H^1(\theta,H_f)/W_i$ making the following
diagram commute:
$$
\xymatrix{
H^1(\sigma,G)\ar@{->}^{\simeq}[r]\ar@{->}[d]_{\simeq}&H^1(\sigma,H_f)/W_i\ar@{->}[d]^{\simeq}\\
H^1(\theta,G)\ar@{->}[r]^{\phi}_{\simeq}&H^1(\theta,H_f)/W_i
}
$$
The top isomorphism is Borovoi's result and  the two vertical arrows
are from Theorem \ref{t:main} applied to $G$ and $H$, respectively. 
\end{proposition}

This is immediate.

\begin{remarkplain}
In an earlier version of this paper we proved the isomorphism
$H^1(\sigma,G)\simeq H^1(\theta,G)$ using this diagram. It is simpler
to prove this isomorphism directly as we have done 
in Section \ref{s:borelserre} and deduce this as a consequence.   
\end{remarkplain}

For later use we note that, in the unequal rank case, the cohomology
is captured by a proper subgroup.

Suppose $H$ is a $\theta$-stable Cartan subgroup. Then $H=TA$ where
$T$ and $A$ are connected complex tori, $T$ is the identity component
of $H^\theta$, and $A$ is the identity component of $H^{-\theta}$.

\begin{corollary}
\label{c:Mf}
Suppose $H_f$ is a $\sigma$ and $\theta$-stable fundamental Cartan
subgroup. Let $A_f$ be the identity component of $H_f^{-\theta}$,
and let  $M_f=\Cent_G(A_f)$. 
Then
$$
H^1(\sigma,G)\simeq H^1(\sigma,M_f)\simeq
H^1(\theta,M_f)\simeq
H^1(\theta,G).
$$
\end{corollary}
Note that $A_f\subset Z\Leftrightarrow M_f=G\Leftrightarrow\text{ the derived group of }G$ is of equal
rank. 

This follows from Proposition \ref{p:borovoi}, and the fact that the
imaginary Weyl groups of $H_f$ in $G$ and $M_f$ are the same. 
\sec{Strong real forms}
\label{s:strong}
In this section we assume $G$ is  connected complex reductive group. 

\begin{lemma}
\label{s:H1Gad}
Fix a real form $\sigma$ of $G$. 
The set of equivalence classes of  real forms in the inner class of
$\sigma$  is parametrized by $H^1(\sigma,\Gad)$. 
\end{lemma}
Explicitly the map is $\cl(h)\mapsto [\int(h)\circ\sigma]$ where
$h\sigma(h)=1$.

By Lemma \ref{l:twist1} it makes sense to define
$H^1([\sigma],\Gad)=H^1(\sigma,\Gad)$. 

\begin{remarkplain}
\label{r:notserre}
Our definition of equivalence of real forms (Definition \ref{d:realform}) is by conjugation by an inner automorphism of $G$.
The standard definition, for example see
\cite{serre_galois}*{III.1}, allows conjugation by $\Aut(G)$. 
With the standard definition the Lemma would hold with
$H^1(\sigma,\Gad)$ replaced by the image of the map to
$H^1(\sigma,\Aut(G))$.

For example suppose $G=\SL(2,\C)\times \SL(2,\C)$. In the inner class
of the split real form of $G$, there are four equivalence classes of
real forms according to our definition:
$\text{split}, \text{compact}$, $\text{split}\times\text{compact}$ and
$\text{compact}\times\text{split}$.  If one allows conjugation by
outer automorphisms there are only three real forms, since the last
two are equivalent.

For simple groups these two notions agree except in type $D_{2n}$. 
See \cite{algorithms}*{Section 3},  \cite{snowbird}*{Example 3.3} and 
Section \ref{s:adjoint}.  
\end{remarkplain}

It is easy to see that two real forms $\sigma_1,\sigma_2$ are in the
same inner class (Definition \ref{d:realform}) if and only if
$\theta_1,\theta_2$ have the same image in $\Out(G)$, where $\theta_i$
is a Cartan involution for $\sigma_i$. So let $\Out(G)_2$ be the set of
elements $\delta$ of $\Out(G)$ such that $\delta^2=1$, 
and write $p$ for the natural map $\Aut(G)\rightarrow \Out(G)$. 
We say a (holomorphic) involution $\theta$ is in the inner class of $\delta$ if 
$p(\theta)=\delta$. We say a real form $\sigma$ is in the inner class of $\delta$
if this holds for  a Cartan involution for $\sigma$.

For example the inner class corresponding to $\delta=1$ is called the
{\em compact} or {\em equal rank} inner class; a real form is in this class if and only if its Cartan 
involution is an inner automorphism.

Suppose $\sigma$ is a real form in the inner class of $\delta$. 
There are two natural choices of a basepoint for the set $H^1(\sigma,\Gad)$ of real forms in this inner class. 
One is the quasisplit (most split) real form.
Because of our focus on $\theta$, rather than $\sigma$,
we prefer to choose the quasicompact (most compact) form, which we now
define.

We say a real form is
{\em quasicompact} if its Cartan involution 
preserves a 
pinning datum 
$(H,B,\{X_\alpha\}_{\alpha \in \Delta})$.
Every inner class contains a unique distinguished involution, which is unique up to conjugation by an inner 
automorphism. See \cite{algorithms}*{Chapter 3}.

\begin{definition}
Suppose $\delta\in\Out(G)_2$. Let $\thetaqc(\delta)$ be a distinguished automorphism 
in the inner class of $\delta$, and let $\sigmaqc(\delta)$ be a corresponding real form 
by Corollary \ref{c:bijection}.
We refer to $[\sigmaqc(\delta)]$ or $[\thetaqc(\delta)]$ as the
equivalence class of quasicompact real forms in the inner class of $\delta$.
\end{definition}
If $\delta$ is fixed we will write $\thetaqc$ and $\sigmaqc$. 

Since any two choices of $\thetaqc(\delta)$ are conjugate by an inner automorphism
$[\thetaqc(\delta)]$ and $[\sigmaqc(\delta)]$ are well defined.

\begin{lemma}
\label{l:canonicalrealforms}
There is a canonical isomorphism 
\begin{equation}
\label{e:isomorphism}
H^1([\sigmaqc],\Gad)\simeq H^1([\thetaqc],\Gad).
\end{equation}
These pointed sets canonically parametrize 
the equivalence classes of real forms in the inner class of $\delta$, 
with the distinguished class going to the equivalence class of
quasicompact  real forms.
\end{lemma}

\begin{exampleplain}
\label{ex:compact}
The group $G(\R)=G^\sigma$ is compact if and only if $\theta=1$. 
Then $H^1(\theta,G)=\{g\in G\mid g^2=1\}/\{g\mapsto xgx\inv\}$, i.e. 
$H^1(\theta,G)$ is 
the set of conjugacy classes of involutions of $G$.
Therefore, if we  fix  a Cartan subgroup $H$,
with Weyl group $W$, then
$$
H^1(\theta,G)\simeq H_2/W
$$
where $H_2=\{h\in H\mid h^2=1\}$.
See Example \ref{ex:serre1}.
\end{exampleplain}

Let $Z=Z(G)$. The action of $\Aut(G)$ on $Z$ factors to an action of
$\Out(G)$ on $Z$. 
Let $\Ztor$ be the subgroup of $Z$ consisting of all elements of finite order.

\begin{lemma}
Fix $\delta\in\Out(G)_2$, and suppose $\sigma$ is a real form in the
inner class of $\delta$.
Let $\theta$ be a Cartan involution for $\sigma$.
Note that the action of $\theta$ on $Z$ coincides with $\delta$.

Then $\Ztor^{\sigma} = \Ztor^{\theta}$ and 
there is a canonical isomorphism
$$
Z^\sigma/(1+\sigma)Z\simeq Z^\delta/(1+\delta)Z.
$$  
\end{lemma}
\begin{proof}
The closure of $\Ztor$ is compact, so by Theorem
\ref{t:conjcpctsubgps} $\Ztor$ is a subgroup of every compact real
form of $G$. Therefore $\sigma^c=\theta\sigma$ acts trivially on $\Ztor$, i.e. 
$\theta,\sigma$ and $\delta$ all have the same action on $\Ztor$. 
Also since $Z^\sigma/(1+\sigma)Z$ and $Z^\delta/(1+\delta)Z$ are two-groups, the quotients $Z/Z^0$, $Z^{\sigma}/(Z^{\sigma})^0$ and $Z^{\delta}/(Z^{\delta})^0$ are finite and $Z^0$ is divisible, it is easy to see
$Z^\sigma/(1+\sigma)Z\simeq 
\Ztor^\sigma/(1+\sigma)\Ztor\simeq 
\Ztor^\delta/(1+\delta)\Ztor\simeq 
Z^\delta/(1+\delta)Z$.
\end{proof}

\begin{definition}
\label{d:zinv}
Fix  $\delta\in\Out(G)_2$ and a real form $\sigma$ in the inner
class of $\delta$. Identify $[\sigma]$ with a class in
$H^1([\sigmaqc],\Gad)$, and define the central invariant 
\begin{equation}
\label{e:centralinvariant}
\zinv([\sigma])\in Z^\delta/(1+\delta)Z
\end{equation}
by the composition of maps:
$$
H^1([\sigmaqc],\Gad)\rightarrow 
H^2(\sigmaqc,Z)\overset\simeq\longrightarrow \wh H^0(\sigmaqc,Z))\overset\simeq\longrightarrow
Z^\sigma/(1+\sigma)Z\overset\simeq\longrightarrow Z^\delta/(1+\delta)Z
$$
The first map is from the connecting homomorphism in \eqref{e:longexact} coming from 
the exact sequence $1\rightarrow Z\rightarrow G\rightarrow
\Gad\rightarrow  1$. The second and third arrows are from
properties of Tate cohomology (see Section \ref{s:prelim}), and the
last one is from the preceding Lemma.
\end{definition}

\begin{remarkplain}
Alternatively we could  define $\zinv:H^1([\thetaqc],\Gad)\rightarrow
Z^\delta/(1+\delta)Z$ similarly, with $\theta,\thetaqc$ in place of
$\sigma,\sigmaqc$.
It is clear from the Lemma and the Definition that the following
diagram commutes:
$$
\xymatrix{
H^1([\sigmaqc],\Gad) \ar@{->}[r]\ar@{->}_{\simeq}[d]&Z^\delta/(1+\delta)Z\\
H^1([\thetaqc],\Gad)\ar@{->}[ur]
}
$$
\end{remarkplain}

The central invariant allows us to see how $H^1(\sigma,G)$ varies in a given inner class,
as in Example \ref{ex:sl2}.
See \cite{serre_galois}*{Section I.5.7, Remark 1}.

\begin{lemma}
\label{l:clarify}
Suppose $\sigma_1,\sigma_2$ are inner forms of $G$. If $\zinv([\sigma_1])=\zinv([\sigma_2])$ then 
$H^1(\sigma_1,G)\simeq H^1(\sigma_2,G)$. 
\end{lemma}

\begin{proof}
Write $\sigma_i=\int(g_i)\circ\sigmaqc$, where $g_i \sigmaqc(g_i)\in Z$ ($i=1,2$).
A straightforward calculation shows that the map $h\rightarrow hg_1g_2\inv$ induces the desired isomorphism, 
provided $g_1\sigmaqc(g_1)=g_2\sigmaqc(g_2)$. Unwinding Definition \ref{d:zinv} we see this condition is equivalent to 
$\zinv(\sigma_1)=\zinv(\sigma_2)$. We leave the details to the reader.
\end{proof}

The map $H^1(\sigma,G)\rightarrow H^1(\sigma,\Gad)$ is not necessarily
surjective. This failure of surjectivity causes some difficulties in
precise statements of the local Langlands conjecture.  See \cite{abv},
\cite{vogan_local_langlands}, and for the
$p$-adic case   \cite{kaletha_rigid}. This leads to the notion of {\it strong real form} of
$G$. 

\begin{definition}
\label{d:strongreal}
Fix $\delta\in\Out(G)_2$ and a distinguished involution $\thetaqc$ in the inner class of $\delta$.
A strong real form, in the inner class of
$\thetaqc$, is an element $g\in G$ satisfying $g\thetaqc(g)\in\Ztor$, i.e.\ an
element of $Z^1(\thetaqc,G;\Ztor)$.
Two strong real forms $g,h$ are said to be equivalent if
$h=tg\thetaqc(t\inv)$ for some $t\in G$.  Write $[g]$ for the
equivalence class of $g$, and let $\SRF_{\thetaqc}(G) = H^1(\thetaqc,G;\Ztor)$ be the set of
equivalence classes of strong real forms in the inner class of
$\thetaqc$.

If $g$ is a strong real form define $\zinv(g)=g\thetaqc(g)\in\Ztor^\delta$. 
We refer to $\zinv$ as the central invariant of a strong real form.
This factors to a well defined map 
$\zinv:\SRF_{\thetaqc}(G)\rightarrow\Ztor^\delta$.
\end{definition}

\begin{remarkplain}
In \cite{algorithms} strong real forms are defined as elements 
of the non-identity component of 
extended group ${}^{\thetaqc} G=G\rtimes\langle\thetaqc\rangle$,
with equivalence being conjugation by $G$. The map taking a strong
real form $g$ of our definition to $g\thetaqc\in G\thetaqc$ is a bijection
between the two notions.   
\end{remarkplain}

We want to eliminate the dependence of $\SRF_{\thetaqc}(G)$ on the choice 
of $\thetaqc$.

\begin{lemma} \label{l:canSRF}
Fix $\delta\in\Out(G)_2$ and distinguished involutions $\thetaqc, \thetaqc'$
in the inner class of $\delta$.
\begin{enumerate}
\item There exists $h \in G$ such that $\thetaqc' = \int(h) \circ \thetaqc \circ
\int(h)^{-1}$, and for any such $h$ we have a bijection
\begin{align*}
Z^1(\thetaqc',G;\Ztor) & \longrightarrow Z^1(\thetaqc,G;\Ztor) \\
g & \longmapsto g h \thetaqc(h)^{-1}
\end{align*}
which is compatible with the maps $\zinv$ to $\Ztor^{\delta}$.
\item The induced map
\[ Z^1(\thetaqc',G;\Ztor) / (1-\delta)Z \rightarrow Z^1(\thetaqc,G;\Ztor) / (1-\delta)Z\]
does not depend on the choice of $h$.
In particular we get a \emph{canonical} bijection $\SRF_{\thetaqc'}(G)
\simeq \SRF_{\thetaqc}(G)$.
\end{enumerate}
\end{lemma}
\begin{proof}
\begin{enumerate}
\item This is an elementary computation.
\item The element $h$ is well defined up to multiplication on the right by an
element of the preimage of $(G_{\mathrm{ad}})^{\thetaqc}$ in $G$. By
\cite{kottwitzStableCusp}*{Lemma 1.6}, this preimage is $Z(G) G^{\thetaqc}$, and the
result follows.
\end{enumerate}
\end{proof}

\begin{definition}
Fix $\delta\in\Out(G)_2$.
Let
$$
\SRF_{\delta}(G) = \lim_{\thetaqc} \SRF_{\thetaqc}(G) 
$$
where the (projective or injective) limit is taken over all
quasicompact involutions $\thetaqc$ in the inner class, using Lemma
\ref{l:canSRF}.
\end{definition}

We have a map $g \mapsto \int(g) \circ \thetaqc$ from $Z^1(\thetaqc, G; \Ztor) /
(1-\delta)Z$ to the set of holomorphic involutions of $G$ in the inner
class of $\delta$, and it is easy to show that it is surjective. Moreover as
$\thetaqc$ varies in the set of distinguished involutions in the inner class of
$\delta$, these maps commute with the maps defined in Lemma \ref{l:canSRF} (1).
We obtain a natural \emph{surjective} map from $\SRF_\delta(G)$ to the set of
equivalence classes of holomorphic involutions of $G$ in the inner class of
$\delta$.

\begin{remarkplain}
\label{rem:equivalent}
In \cite{abv} and 
\cite{vogan_local_langlands} strong real forms are defined 
in terms of the Galois action, as opposed to the Cartan involution 
as  in \cite{algorithms}  (and elsewhere, including \cite{bowdoin}).
The preceding discussion together with Corollary \ref{c:main} show that these two theories are indeed equivalent. 
However the choices of basepoints in the two theories are different. 
In the Galois setting  we choose the quasisplit form, and 
in the algebraic setting  we use the quasicompact one. 

The invariant of a Galois strong real form 
is defined  \cite{vogan_local_langlands}*{(2.8)(c)}.
This differs from the normalization here by multiplication 
by  $\exp(2\pi i\ch\rho)\in Z$.
Note that 
the ``pure'' rational forms, which are parametrized by $H^1(\sigma,G)$,
include the quasisplit one 
\cite{vogan_local_langlands}*{Proposition 2.7(c)}, rather than the quasicompact one.
\end{remarkplain}

We can now describe strong real forms in terms of Galois cohomology sets  $H^1(\sigma,G)$.
Recall if $\sigma$ is a real form in the inner class of $\delta$, then
the central invariant  $\zinv([\sigma])$  is an element of 
$ Z^\delta/(1+\delta)Z$ (Definition \ref{d:zinv}).

\begin{proposition}
\label{p:H1}
Suppose $\sigma$ is a real form of $G$, in the inner class of
$\delta$. Choose a representative 
$z\in \Ztor^\delta$ of $\zinv([\sigma])\in Z^\delta/(1+\delta)Z$. 
Then there is a bijection
$$
H^1(\sigma,G)\longleftrightarrow\text{the set of strong real forms of central invariant }z.
$$
\end{proposition}
\begin{proof}
Fix a distinguished involution $\thetaqc$ of $G$ in the inner class of $\delta$.
There exists $g \in G$ such that $\int(g) \circ \thetaqc$ is a Cartan involution
for $\sigma$ and $g \thetaqc(g) = z$.
Fix such a $g$ and let $\theta = \int(g) \circ \thetaqc$.
Then we have a bijection
\begin{align*}
H^1(\thetaqc, G; \{z\}) & \longrightarrow H^1(\theta, G) \\
h & \longmapsto h g^{-1}
\end{align*}
and composing with the isomorphism $H^1(\theta, G) \simeq H^1(\sigma, G)$ of
Corollary \ref{c:main} gives the result.
\end{proof}

Note that the bijection not only depends on the choice of representative $z \in
\Ztor^\delta$ of $\zinv([\sigma])\in \Ztor^\delta/(1+\delta)\Ztor$, but also on
the choice of $g$ in the proof: $g$ could be replaced by $g x$, where $x \in Z$
is such that $x \delta(x) = 1$.

\begin{corollary}
\label{c:interpret}
Choose representatives $\{z_i\mid i\in I\}$ 
for the image of $\zinv:\SRF_\delta(G)\rightarrow \Ztor^\delta$. 
For each $i\in I$ choose a real form $\sigma_i$ of $G$ such that $\zinv([\sigma_i])=z_i \mod (1+\delta)\Ztor$. Then 
there is a  bijection
$$
\SRF_\delta(G)\longleftrightarrow\bigcup_i H^1(\sigma_i,G).
$$
\end{corollary}
This gives an interpretation of $\SRF_\delta(G)$ in classical cohomological terms.
A similar statement holds in the p-adic case \cite{kaletha_rigid}.

The set $I$ is finite if and only if the connected center of $G$ is split (this
condition only depends on $\delta$). 
As in \cite{kaletha_rigid}  
or \cite{algorithms}*{Section 13} the theory can be modified to replace this with a finite set even when this condition is not satisfied.
In any case the group $\Ztor^\delta / (1+\delta)\Ztor$ is finite, and for $z \in
\Ztor^\delta$ and $x \in \Ztor$ there is an obvious isomorphism
\[ H^1(\thetaqc, G; \{z\}) \simeq H^1(\thetaqc, G; \{z x \delta(x)\}). \]

\begin{corollary}
\label{c:equalrank}
Suppose $\sigma$ is an equal rank real form of $G$.
Choose $x\in G$ so that $\int(x)$ is a Cartan involution for $\sigma$, and let
$z=x^2\in Z$. Then
$$
H^1(\sigma,G)\longleftrightarrow \text{the set of conjugacy classes
  of $G$ with square equal to $z$}
$$
If $H$ is a Cartan subgroup, with Weyl group $W$, then this is equal to
\begin{equation}
\label{e:h2z}
\{h\in H\mid h^2=z\}/W
\end{equation}
\end{corollary}

\begin{exampleplain}
\label{ex:serre2}
Taking $x=z=I$ gives $G(\R)$ compact and recovers \cite{serre_galois}*{III.4.5}: $H^1(\sigma,G)$ is the set of  conjugacy classes 
of involutions in  $G$. See Example \ref{ex:serre1}.
\end{exampleplain}

\begin{exampleplain}
Let $G(\R)=Sp(2n,\R)$. We can take $x=\diag(iI_n,-iI_n)$, $z=-I$. 
It is easy to see that every element of $G$ whose square is $-I$ is conjugate to $x$.
This
gives the classical result $H^1(\sigma,G)=1$, which 
is equivalent to the classification of nondegenerate symplectic forms \cite{platonov_rapinchuk}*{Chapter 2}.
\end{exampleplain}

\begin{exampleplain}
Suppose $G(\R)=SO(Q)$, the isometry group of a nondegenerate real quadratic
form.
Suppose $Q$ has signature $(p,q)$.
If $pq$ is even we can take $z=I$, Corollary \ref{c:equalrank} applies, 
and the set \eqref{e:h2z} is equal to $\{\diag(I_r,-I_s)\mid r+s=p+q;\,s\text{ even}\}$.

Suppose $p$ and $q$ are odd. Apply Corollary \ref{c:Mf} with  $M_f(\R)=SO(p-1,q-1)\times GL(1,\R)$.
By the previous case we conclude $H^1(\sigma,G)$ is parametrized by $\{\diag(I_r,I_s)\mid r+s=p+q-2;\, r,s\text{ even}\}$. 
Adding $(1,1)$ this is the same as $\{\diag(I_r,-I_s)\mid r+s=p+q;\, s\text{ odd}\}$. 

In all cases we recover the classical fact that $H^1(\sigma,G)$ parametrizes the set of equivalence classes of 
quadratic forms of the same dimension and discriminant as $Q$ \cite{platonov_rapinchuk}*{Chapter 2},
\cite{serre_galois}*{III.3.2}.
\end{exampleplain}

\begin{exampleplain}
\label{ex:spin}
Now suppose $G(\R)=\Spin(p,q)$, which is  a (connected) two-fold cover of
the identity component of 
$\SO(p,q)$. 
A  calculation similar to that in the previous example 
shows that 
$|H^1(\sigma,\Spin(p,q))|=\lfloor\frac{p+q}4\rfloor+\delta(p,q)$
where $0\le \delta(p,q)\le 3$ depends on $p,q\pmod4$.
See Section \ref{s:simply}.

Skip Garibaldi pointed out this result can also be derived from 
the exact cohomology sequence associated to 
the exact sequence $1\rightarrow\Ztwo\rightarrow \Spin(n,\C)\rightarrow SO(n,\C)\rightarrow 1$;
the preceding result; the fact that $SO(p,q)$ 
is connected if $pq=0$ and otherwise has two connected components; 
and a calculation of the image of the map from
$H^1(\sigma,\Spin(n,\C))\rightarrow H^1(\sigma,SO(n,\C))$.
See \cite{bookofinvolutions}*{after (31.41)}, \cite{serre_galois}*{III.3.2} and also section \ref{s:fibers}.
The result is:

\medskip

\noindent $|H^1(\sigma,\Spin(Q))|$ equals the number of  quadratic forms having the
same dimension, discriminant, and Hasse invariant as $Q$ with each (positive
or negative) definite form counted twice.
\end{exampleplain}

\begin{remarkplain}
Kottwitz  relates $H^1(\sigma,G)$ to
the center of the dual group \cite{kottwitzStable}*{Theorem 1.2}. This is a somewhat different type of
result. It describes a certain quotient $H^1_{\mathrm{sc}}(\sigma, G)$ of $H^1(\sigma,G)$ (see \cite{kaletha_rigid}*{§3.4}), but if $G$ is simply connected this gives no information.
\end{remarkplain}

\sec{Fibers of $H^1(\sigma,G)\rightarrow H^1(\sigma,\Gbar)$}
\label{s:fibers}

In this section $G$ is a connected complex reductive group, and $\sigma$ is a real form of $G$.
Suppose $A\subset Z(G)$ is $\sigma$-stable and let $\Gbar=G/A$.
It is helpful to analyze
the fibers of the map $\psi:H^1(\sigma,G) \rightarrow H^1(\sigma,\Gbar)$.
In particular taking $G=\Gsc, \Gbar=\Gad$, 
and summing over $H^1(\sigma,\Gad)$, we obtain a description 
of $H^1(\sigma,\Gsc)$, complementary  to that of Proposition \ref{p:H1}.

Write $G(\R,\sigma)=G^\sigma$ and $\Gbar(\R,\sigma)=\Gbar^\sigma$.
Write $p$ for the projection map $G\rightarrow \Gbar$. This restricts to a map
$G(\R,\sigma)\rightarrow \Gbar(\R,\sigma)$,  taking the identity component
of $G(\R,\sigma)$ to that of $\Gbar(\R,\sigma)$. Therefore $p$  factors to a map
(not necessarily an injection):
\begin{subequations}
\renewcommand{\theequation}{\theparentequation)(\alph{equation}}  
\begin{equation}
p^*:\pi_0(G(\R,\sigma))\rightarrow \pi_0(\Gbar(\R,\sigma)).
\end{equation}
Define
\begin{equation}
\pi_0(G,\Gbar,\sigma)=\pi_0(\Gbar(\R,\sigma))/p^*(\pi_0(G(\R,\sigma))).
\end{equation}

There is a natural action of  $\overline G(\R,\sigma)$ on  $H^1(\sigma,A)$ defined as follows.
Suppose $g\in \Gbar(\R,\sigma)$. Choose $h\in G$ satisfying $p(h)=g$.
Then $g:a\rightarrow ha\sigma(h\inv)$ factors to a well defined action of 
$\Gbar(\R,\sigma)$ on $H^1(\sigma,A)$. Furthermore
the image of $G(\R,\sigma)$, which includes the identity component,
acts trivially, so this factors to an  action of $\pi_0(G,\Gbar,\sigma)$. 

\end{subequations}

\begin{proposition}
\label{p:fiber}
Suppose $\gamma\in H^1(\sigma,G)$, and write $\gamma=\cl(g)$ 
($g\in G^{-\sigma}$). Let $\sigma_\gamma=\int(g)\circ\sigma$.
Then there is a bijection
$$
H^1(\sigma,G)\supset \psi\inv(\psi(\gamma))\longleftrightarrow H^1(\sigma,A)/\pi_0(G,\Gbar,\sigma_\gamma).
$$

\end{proposition}

\begin{proof}
First assume $\gamma$ is trivial, and take $g=1$. 
Consider the exact sequence
$$
H^0(\sigma,G)\rightarrow H^0(\sigma,\Gbar)\rightarrow H^1(\sigma,A)\overset{\phi}\rightarrow H^1(\sigma,G)\overset\psi\rightarrow H^1(\sigma,\Gbar).
$$
This says $\psi\inv(\psi((\gamma))=\phi(H^1(\sigma,A))$, i.e. the orbit of the group $H^1(\sigma,A)$ acting on 
the identity coset.
This  is $H^1(\sigma,A)$, modulo the action of $H^0(\sigma,\Gbar)$,
and this action factors through the image of $H^0(\sigma,G)$. 
The general case follows from an easy twisting argument.
\end{proof}

We specialize to the case $G=\Gsc$ is simply connected and $\overline G=\Gad=\Gsc/\Zsc$ is the adjoint group.

\begin{corollary}
\label{c:fiber} 
Suppose $\sigma$ is a real form of $\Gsc$ and consider the map $\psi:H^1(\sigma,\Gsc)\rightarrow H^1(\sigma,\Gad)$. 

Suppose $\gamma\in H^1(\sigma,\Gad)$, and write $\gamma=\cl(g)$ ($g\in\Gad^{-\sigma}$). 
Let $\sigma_\gamma=\int(g)\circ\sigma$, viewed as an involution of $\Gsc$.
\begin{subequations}
\renewcommand{\theequation}{\theparentequation)(\alph{equation}}  
\begin{equation}
	\gamma\text{ is in the image of }\psi\Leftrightarrow \zinv([\sigma_\gamma])=\zinv([\sigma]),
\end{equation}
in which case
\begin{equation}
|\psi\inv(\gamma)|=|H^1(\sigma,\Zsc)|/|\pi_0(\Gad(\R,\sigma_\gamma))|.
\end{equation}
Furthermore
\begin{equation}
	|H^1(\sigma,\Gsc)|=|H^1(\sigma,\Zsc)|\sum_{\substack{\gamma\in H^1(\sigma,\Gad)\\ \zinv([\sigma_\gamma])=\zinv([\sigma])}}|\pi_0(\Gad(\R),\sigma_\gamma)|\inv
\end{equation}
\end{subequations}
\end{corollary}

\begin{proof}
Statements (b) and (c) follow from the Proposition. 
For (a), when $\sigma=\sigmaqc$ and $\gamma=1$ the proof is immediate, and the general case follows by twisting. 
We leave the details to the reader.
\end{proof}

\sec{Tables}
\label{s:tables}

Most of these results can be computed by hand from Theorem \ref{t:H1}, 
or using Proposition \ref{p:fiber} and the classification of real forms (i.e. the adjoint case). 

By Theorem \ref{t:H1} the computation of $H^1(\Gamma,G)$ reduces to
calculating the strong real forms of $G$ and their central invariants.
The Atlas of Lie Groups and Representations does this computation as
part of its parametrization of (strong) real forms. This comes down to
calculating the orbits of a finite group (a subgroup of the Weyl
group) on a finite set (related to elements of order $2$ in a Cartan subgroup). 
See \cite{algorithms}*{Proposition 12.9} and 
{\tt www.liegroups.org/tables/galois}.

\subsec{Classical groups}
\label{s:classical}

\begin{tabular}{|l|l|l|}
\hline
Group&$|H^1(\sigma,G)|$& \\
\hline
$SL(n,\R),GL(n,\R)$ & $1$ &\\  
\hline
$SU(p,q)$ & $ \lfloor\frac p2\rfloor+\lfloor\frac q2\rfloor+1$
&$\begin{gathered}
\text{Hermitian forms of rank $p+q$ and }\\\text{discriminant $(-1)^q$}\end{gathered}$ \\\hline
$SL(n,\mathbb H)$ & $2$ & $\R^*/\text{Nrd}_{\H/\R}(\H^*)$ \\\hline
$Sp(2n,\R)$ & $1$ & real symplectic forms of rank $2n$\\\hline
$Sp(p,q)$ & $p+q+1$ & quaternionic Hermitian forms of rank $p+q$\\\hline
$SO(p,q)$ & 
$\lfloor \frac p2\rfloor+\lfloor \frac q2\rfloor+1$
&
$
\begin{gathered}
\text{real symmetric bilinear forms of rank $n$}\\\text{and discriminant $(-1)^q$}
\end{gathered}
$
\\\hline
$SO^*(2n)$ & 2 &\\
\hline
\end{tabular}

\bigskip

Here $\H$ is the quaternions, and $\text{Nrd}_{\H/\R}$ is the reduced norm map from
$\H^*$ to $\R^*$ (see \cite{platonov_rapinchuk}*{Lemma 2.9}).
For more information on Galois cohomology of classical groups 
see
\cite{serre_galois}, \cite{platonov_rapinchuk}*{Sections 2.3 and 6.6} and \cite{bookofinvolutions}*{Chapter VII}.

\subsec{Simply connected groups}
\label{s:simply}

The only  simply connected  groups with classical root system, which are not  
in the table in Section \ref{s:classical} 
are $\Spin(p,q)$ and $\Spin^*(2n)$.

Define $\delta(p,q)$ by the following table, depending on $p,q\pmod 4$.

$$
\begin{tabular}{c|cccc}
&0&1&2&3\\\hline
0&3&2&2&2\\
1&2&1&1&0\\
2&2&1&0&0\\
3&2&0&0&0  
\end{tabular}
$$

See Example \ref{ex:spin} for an explanation of these numbers.
\bigskip

\centerline{\begin{tabular}{|l|l|}
\hline
Group&$|H^1(\sigma,G)|$\\\hline
$\Spin(p,q)$& $\lfloor\frac{p+q}4\rfloor +\delta(p,q)$\\  
\hline
$\Spin^*(2n)$& $2$\\\hline
\end{tabular}}

\bigskip
\bigskip

\begin{tabular}{|c|c|c|c|c|c|}
\hline
\multicolumn{6}{|c|}{Simply connected exceptional groups}\\
\hline
inner class&group&$K$&real rank&name&$|H^1(\sigma,G)|$\\
\hline
compact&$E_6$ &$A_5A_1$ & $4$ & $\begin{gathered}\text{quasisplit'}\\ {\text{quaternionic}}\end{gathered} $& 3\\\hline
&$E_6$ &$D_5T$ & $2$ & Hermitian & 3\\\hline
&$E_6$ &$E_6$ & $0$ & compact & 3\\\hline
split&$E_6$ &$C_4$ & $6$ & split & 2\\\hline
&$E_6$ &$F_4$ & $2$ & quasicompact & 2\\\hline\hline
compact& $E_7$ & $A_7$ & $7$ &split & $2 $\\\hline
& $E_7$ & $D_6A_1$ & $4$ &quaternionic & $4$\\\hline
& $E_7$ & $E_6T$ & $3$ &Hermitian & $2$\\\hline
& $E_7$ & $E_7$ & $0$ &compact & $4$\\\hline
\hline
compact& $E_8$ & $D_8$ & $8$ &split & $3$\\\hline
& $E_8$ & $E_7A_1$ & $4$ &quaternionic & $3$\\\hline
& $E_8$ & $E_8$ & $0$ &compact & $3$\\\hline
\hline
compact& $F_4$ & $C_3A_1$ & $4$ &split & $3$\\\hline
& $F_4$ & $B_4$ & $1$ & & $3$\\\hline
& $F_4$ & $F_4$ & $0$ &compact & $3$\\\hline
\hline
compact& $G_2$ & $A_1A_1$ & $2$ &split & $2$\\\hline
& $G_2$ & $G_2$ & $0$ & compact & $2$\\\hline
\end{tabular}

\newpage

\subsec{Adjoint groups}
\label{s:adjoint}

If $G$ is adjoint $|H^1(\sigma,G)|$ is the number of real forms 
in the given inner class, which is well known. 
We also include the  component group, which is useful in connection with 
Corollary  \ref{c:fiber}.

One technical point arises in the case  of $PSO^*(2n)$. If $n$ is even there are 
two real forms which are related by an outer, but not an inner, automorphism. 
See Remark \ref{r:notserre}.

\bigskip

\begin{tabular}{|c|c|c|}
\hline
\multicolumn{3}{|c|}{Adjoint classical groups}\\
\hline
Group&$|\pi_0(G(\R))|$&$|H^1(\sigma,G)|$ \\
\hline
$PSL(n,\R)$  &  $
\begin{cases}
  2&n\text{ even}\\
  1&n\text{ odd}\\
\end{cases}$
& $
\begin{cases}
  2&n\text{ even}\\
  1&n\text{ odd}\\
\end{cases}$
\\\hline
$PSL(n,\H)$ &$1$&$2$\\\hline
$PSU(p,q)$ &$
\begin{cases}
  2&p=q\\1&\text{ otherwise}
\end{cases}
$
 & $\lfloor \frac{p+q}2\rfloor+1$\\\hline
$PSO(p,q)$ &
$
\begin{cases}
1&pq=0\\
1&p,q\text{ odd and } p\ne q\\
4&p=q\text{ even}\\
2&\text{otherwise}  
\end{cases}
$
&$
\begin{cases}
\lfloor\frac{p+q+2}4\rfloor&p,q\text{ odd}\\  
\frac{p+q}4+3&p,q\text{ even}, p+q=0\pmod4\\
\frac{p+q-2}4+2&p,q\text{ even}, p+q=2\pmod4\\
\frac{p+q+1}2&p,q\text{ opposite parity}\\  
\end{cases}
$
\\\hline
$PSO^*(2n)$ &
$\begin{cases}
  2&n\text{ even}\\
  1&n\text{ odd}\\
\end{cases}
$
&$
\begin{cases}
\frac n2+3&n\text{ even}  \\
\frac {n-1}2+2&n\text{ odd}  
\end{cases}
$
\\\hline

$PSp(2n,\R)$ &$2$&$\lfloor \frac n2\rfloor +2$\\\hline
$PSp(p,q)$ &
$
\begin{cases}
2&p=q\\
1&else  
\end{cases}
$
&$\lfloor \frac {p+q}2\rfloor +2$\\\hline
\end{tabular}

\bigskip

The groups $E_8,F_4$ and $G_2$ are both simply connected and adjoint. 
Furthermore in type $E_6$ the center of 
the simply connected group $\Gsc$ has order $3$, and 
it follows that $H^1(\sigma,\Gad)=H^1(\sigma,\Gsc)$ in these cases.
So the only groups not covered by the table in Section \ref{s:simply} 
are adjoint groups of type $E_7$.

\medskip

\begin{tabular}{|c|c|c|c|c|c|c|}
\hline
\multicolumn{7}{|c|}{Adjoint exceptional groups}\\
\hline
inner class&group&$K$&real rank&name&$\pi_0(G(\R))$ &$|H^1(G)|$\\
\hline
compact& $E_7$ & $A_7$ & $7$ &split & 2&$4$\\\hline
& $E_7$ & $D_6A_1$ & $4$ &quaternionic & $1$ &$4$\\\hline
& $E_7$ & $E_6T$ & $3$ &Hermitian & $2$ &$4$\\\hline
& $E_7$ & $E_7$ & $0$ &compact & $1$&$4$\\\hline
\end{tabular}
 
\bibliographystyle{plain}


\begin{bibdiv}
\begin{biblist}

\bib{bowdoin}{incollection}{
      author={Jeffrey Adams},
       title={Lifting of characters},
        date={1991},
   booktitle={Harmonic analysis on reductive groups ({B}runswick, {ME}, 1989)},
      series={Progr. Math.},
      volume={101},
   publisher={Birkh\"auser Boston},
     address={Boston, MA},
       pages={1\ndash 50},
      review={\MR{MR1168475 (93d:22014)}},
}

\bib{snowbird}{incollection}{
      author={Jeffrey Adams},
       title={Guide to the atlas software: Computational representation theory
  of real reductive groups},
        date={2008},
   booktitle={Representation theory of real reductive groups, proceedings of
  conference at snowbird, july 2006},
      series={Contemp. Math.},
   publisher={Amer. Math. Soc.},
}

\bib{abv}{book}{
      author={Jeffrey Adams},
      author={Dan Barbasch},
      author={David~A. Vogan Jr.},
       title={The {L}anglands classification and irreducible characters for
  real reductive groups},
      series={Progress in Mathematics},
   publisher={Birkh\"auser Boston Inc.},
     address={Boston, MA},
        date={1992},
      volume={104},
        ISBN={0-8176-3634-X},
      review={\MR{MR1162533 (93j:22001)}},
}

\bib{algorithms}{article}{
      author={Jeffrey Adams},
      author={Fokko du~Cloux},
       title={Algorithms for representation theory of real reductive groups},
        date={2009},
        ISSN={1474-7480},
     journal={J. Inst. Math. Jussieu},
      volume={8},
      number={2},
       pages={209\ndash 259},
      review={\MR{MR2485793}},
}

\bib{av1}{article}{
      author={Jeffrey Adams},
      author={David~A. Vogan Jr.},
       title={{$L$}-groups, projective representations, and the {L}anglands
  classification},
        date={1992},
        ISSN={0002-9327},
     journal={Amer. J. Math.},
      volume={114},
      number={1},
       pages={45\ndash 138},
      review={\MR{MR1147719 (93c:22021)}},
}

\bib{borelbook}{book}{
    author = {Armand Borel},
     title = {Linear algebraic groups},
    series = {Graduate Texts in Mathematics},
    volume = {126},
   edition = {Second},
 publisher = {Springer-Verlag, New York},
      year = {1991},
     pages = {xii+288}
}

\bib{borel bourbaki}{incollection}{
    author = {Armand Borel},
     title = {Sous-groupes compacts maximaux des groupes de {L}ie},
 booktitle = {S\'eminaire {B}ourbaki, {V}ol.\ 1},
     pages = {Exp.\ No.\ 33, 271--279},
 publisher = {Soc. Math. France, Paris},
      year = {1995}
}

\bib{borelserre}{article}{
      author={Armand Borel},
      author={Jean-Pierre Serre},
       title={Th\'eor\`emes de finitude en cohomologie galoisienne},
        date={1964},
        ISSN={0010-2571},
     journal={Comment. Math. Helv.},
      volume={39},
       pages={111\ndash 164},
      review={\MR{0181643}},
}

\bib{borovoi}{article}{
      author={Mikhail Borovoi},
       title={Galois cohomology of real reductive groups and real forms of
  simple lie algebras},
        date={1988},
     journal={Funct. Anal. Appl.},
      volume={22},
      number={2},
       pages={135\ndash 136},
}

\bib{borovoi_timashev}{article}{
      author={Mikhail Borovoi},
      author={Dmitry A. Timashev},
       title={Galois cohomology of real semisimple groups}
        date={2015},
        note={arXiv:1506.06252}
}

\bib{neronmodels}{book}{
    author = {Siegfried Bosch},
    author={ Werner L\"utkebohmert},
    author={Michel  Raynaud},
     title = {N\'eron models},
    series = {Ergebnisse der Mathematik und ihrer Grenzgebiete (3) [Results
              in Mathematics and Related Areas (3)]},
    volume = {21},
 publisher = {Springer-Verlag, Berlin},
      year = {1990},
     pages = {x+325}
}

\bib{bourbakiLie78}{book}{
      author={Nicolas Bourbaki},
       title={Lie groups and {L}ie algebras. {C}hapters 7--9},
      series={Elements of Mathematics (Berlin)},
   publisher={Springer-Verlag, Berlin},
        date={2005},
        ISBN={3-540-43405-4},
        note={Translated from the 1975 and 1982 French originals by Andrew
  Pressley},
      review={\MR{2109105}},
}

\bib{hochschild}{book}{
      author={Gerhard Hochschild},
       title={The structure of {L}ie groups},
   publisher={Holden-Day, Inc., San Francisco-London-Amsterdam},
        date={1965},
      review={\MR{0207883}},
}

\bib{kaletha_rigid}{article}{
      author={Tasho Kaletha},
       title={Rigid inner forms of real and p-adic groups},
       journal = {Ann. of Math. (2)},
       volume = {184},
       year = {2016},
       number = {2},
       pages = {559--632}
}

\bib{bookofinvolutions}{book}{
      author={Max-Albert Knus},
      author={Alexander Merkurjev},
      author={Markus Rost},
      author={Jean-Pierre Tignol},
       title={The book of involutions},
      series={American Mathematical Society Colloquium Publications},
   publisher={American Mathematical Society},
     address={Providence, RI},
        date={1998},
      volume={44},
        ISBN={0-8218-0904-0},
        note={With a preface in French by J. Tits},
      review={\MR{1632779 (2000a:16031)}},
}

\bib{kostant_rallis}{article}{
      author={Bertram Kostant},
      author={Stephen Rallis},
       title={Orbits and representations associated with symmetric spaces},
        date={1971},
        ISSN={0002-9327},
     journal={Amer. J. Math.},
      volume={93},
       pages={753\ndash 809},
      review={\MR{0311837}},
}

\bib{kottwitzStableCusp}{article}{
      author={Robert~E. Kottwitz},
       title={Stable trace formula: cuspidal tempered terms},
        date={1984},
     journal={Duke Math. J.},
      volume={51},
      number={3},
       pages={611\ndash 650},
}

\bib{kottwitzStable}{article}{
      author={Robert~E. Kottwitz},
       title={Stable trace formula: elliptic singular terms},
        date={1986},
        ISSN={0025-5831},
     journal={Math. Ann.},
      volume={275},
      number={3},
       pages={365\ndash 399},
         url={http://dx.doi.org/10.1007/BF01458611},
      review={\MR{858284 (88d:22027)}},
}

\bib{matsuki}{article}{
      author={Toshiko Matsuki},
       title={The orbits of affine symmetric spaces under the action of minimal
  parabolic subgroups},
        date={1979},
        ISSN={0025-5645},
     journal={J. Math. Soc. Japan},
      volume={31},
      number={2},
       pages={331\ndash 357},
         url={http://dx.doi.org/10.2969/jmsj/03120331},
      review={\MR{527548 (81a:53049)}},
}

\bib{mostow}{article}{
      author={George~D. Mostow},
       title={Self-adjoint groups},
        date={1955},
        ISSN={0003-486X},
     journal={Ann. of Math. (2)},
      volume={62},
       pages={44\ndash 55},
      review={\MR{0069830}},
}

\bib{steinberg}{book}{
      author={Robert Steinberg},
       title={Endomorphisms of linear algebraic groups},
      series={Memoirs of the American Mathematical Society, No. 80},
   publisher={American Mathematical Society, Providence, R.I.},
        year={1968},
       pages={108}
}

\bib{ov}{book}{
      author={Arkadij~L. Onishchik},
      author={Ernest~B. Vinberg},
       title={Lie groups and algebraic groups},
      series={Springer Series in Soviet Mathematics},
   publisher={Springer-Verlag},
     address={Berlin},
        date={1990},
        ISBN={3-540-50614-4},
        note={Translated from the Russian and with a preface by D. A. Leites},
      review={\MR{91g:22001}},
}

\bib{platonov_rapinchuk}{book}{
      author={Vladimir Platonov},
      author={Andrei Rapinchuk},
       title={Algebraic groups and number theory},
      series={Pure and Applied Mathematics},
   publisher={Academic Press Inc.},
     address={Boston, MA},
        date={1994},
      volume={139},
        ISBN={0-12-558180-7},
        note={Translated from the 1991 Russian original by Rachel Rowen},
      review={\MR{MR1278263 (95b:11039)}},
}


\bib{sekiguchi_correspondence}{article}{
      author={Sekiguchi, Jir{\=o}},
       title={Remarks on real nilpotent orbits of a symmetric pair},
        date={1987},
        ISSN={0025-5645},
     journal={J. Math. Soc. Japan},
      volume={39},
      number={1},
       pages={127\ndash 138},
         url={http://dx.doi.org/10.2969/jmsj/03910127},
      review={\MR{867991}},
}

\bib{serre_galois}{book}{
      author={Jean-Pierre Serre,},
       title={Galois cohomology},
     edition={English},
      series={Springer Monographs in Mathematics},
   publisher={Springer-Verlag},
     address={Berlin},
        date={2002},
        ISBN={3-540-42192-0},
        note={Translated from the French by Patrick Ion and revised by the
  author},
      review={\MR{1867431 (2002i:12004)}},
}

\bib{springer_book}{book}{
      author={Tonny~A. Springer},
       title={Linear algebraic groups},
     edition={Second},
      series={Progress in Mathematics},
   publisher={Birkh\"auser Boston Inc.},
     address={Boston, MA},
        date={1998},
      volume={9},
        ISBN={0-8176-4021-5},
      review={\MR{MR1642713 (99h:20075)}},
}

\bib{vogan_green}{book}{
      author={David~A. Vogan Jr.},
       title={Representations of real reductive {L}ie groups},
      series={Progress in Mathematics},
   publisher={Birkh\"auser Boston},
     address={Mass.},
        date={1981},
      volume={15},
        ISBN={3-7643-3037-6},
      review={\MR{MR632407 (83c:22022)}},
}

\bib{ic4}{article}{
      author={David~A. Vogan Jr.},
       title={Irreducible characters of semisimple {L}ie groups. {IV}.
  {C}haracter-multiplicity duality},
        date={1982},
        ISSN={0012-7094},
     journal={Duke Math. J.},
      volume={49},
      number={4},
       pages={943\ndash 1073},
      review={\MR{MR683010 (84h:22037)}},
}

\bib{vogan_local_langlands}{incollection}{
      author={David~A. Vogan Jr.},
       title={The local {L}anglands conjecture},
        date={1993},
   booktitle={Representation theory of groups and algebras},
      series={Contemp. Math.},
      volume={145},
   publisher={Amer. Math. Soc.},
     address={Providence, RI},
       pages={305\ndash 379},
      review={\MR{MR1216197 (94e:22031)}},
}

\bib{warner_I}{book}{
      author={Garth Warner},
       title={Harmonic analysis on semi-simple {L}ie groups. {II}},
   publisher={Springer-Verlag},
     address={New York},
        date={1972},
        note={Die Grundlehren der mathematischen Wissenschaften, Band 189},
      review={\MR{0499000 (58 \#16980)}},
}


\end{biblist}
\end{bibdiv}

\def\cprime{$'$} \def\cftil#1{\ifmmode\setbox7\hbox{$\accent"5E#1$}\else
  \setbox7\hbox{\accent"5E#1}\penalty 10000\relax\fi\raise 1\ht7
  \hbox{\lower1.15ex\hbox to 1\wd7{\hss\accent"7E\hss}}\penalty 10000
  \hskip-1\wd7\penalty 10000\box7}
  \def\cftil#1{\ifmmode\setbox7\hbox{$\accent"5E#1$}\else
  \setbox7\hbox{\accent"5E#1}\penalty 10000\relax\fi\raise 1\ht7
  \hbox{\lower1.15ex\hbox to 1\wd7{\hss\accent"7E\hss}}\penalty 10000
  \hskip-1\wd7\penalty 10000\box7}
  \def\cftil#1{\ifmmode\setbox7\hbox{$\accent"5E#1$}\else
  \setbox7\hbox{\accent"5E#1}\penalty 10000\relax\fi\raise 1\ht7
  \hbox{\lower1.15ex\hbox to 1\wd7{\hss\accent"7E\hss}}\penalty 10000
  \hskip-1\wd7\penalty 10000\box7}
  \def\cftil#1{\ifmmode\setbox7\hbox{$\accent"5E#1$}\else
  \setbox7\hbox{\accent"5E#1}\penalty 10000\relax\fi\raise 1\ht7
  \hbox{\lower1.15ex\hbox to 1\wd7{\hss\accent"7E\hss}}\penalty 10000
  \hskip-1\wd7\penalty 10000\box7} \def\cprime{$'$} \def\cprime{$'$}
  \def\cprime{$'$} \def\cprime{$'$} \def\cprime{$'$} \def\cprime{$'$}
  \def\cprime{$'$} \def\cprime{$'$}

\enddocument
\end